\documentclass[smallextended]{svjour3} 
\smartqed 
\usepackage{graphicx}    
\RequirePackage{fix-cm}
\usepackage{stmaryrd}
\usepackage{amssymb} 
\usepackage{amsmath}
\usepackage{amsfonts}
\usepackage{bbm} 
\usepackage{mathrsfs}
\usepackage{bigints} 
\usepackage{enumerate} 
\usepackage{authblk} 
\usepackage{color} 

\newcommand\at[2]{\left.#1\right|_{#2}} 

\begin{document}

\title{Mean Field Games with a Dominating Player}

\author{A. Bensoussan 
		\and M. H. M. Chau 
		\and S. C. P. Yam
}
\date{Received: date / Accepted: date}
\institute{A. Bensoussan \at
International Center for Decision and Risk Analysis,Jindal  School of Management, The University of Texas at Dallas\\
Department of Systems Engineering and Engineering Management, College of Science and Engineering, City University of Hong Kong\\
\email{axb046100@utdallas.edu}           
           \and
M. H. M. Chau, S. C. P. Yam\at
Department of Statistics, The Chinese University of Hong Kong\\
\email{michaelchaumanho@gmail.com, scpyam@sta.cuhk.edu.hk}
}
\newtheorem{thm}{Theorem}[section]
\newtheorem{cor}[thm]{Corollary}
\newtheorem{defn}[thm]{Definition}
\newtheorem{lem}[thm]{Lemma}
\newtheorem{prop}[thm]{Proposition}
\newtheorem{pblm}[thm]{Problem}
\spnewtheorem*{rk}{Remark}{\bf}{\it}

\setcounter{section}{0}
\maketitle

\begin{abstract}
\noindent In this article, we consider mean field games between a dominating player and a group of representative agents, each of which acts similarly and also interacts with each other through a mean field term being substantially influenced by the dominating player. We first provide the general theory and discuss the necessary condition for the optimal controls and equilibrium condition by adopting adjoint equation approach. We then present a special case in the context of linear-quadratic framework, in which a necessary and sufficient condition can be asserted by stochastic maximum principle; we finally establish the sufficient condition that guarantees the unique existence of the equilibrium control. The proof of the convergence result of finite player game to mean field counterpart is provided in Appendix.\\

\end{abstract}

\section{Introduction}
For long, modeling the joint interactive behaviour of individual objects(agents) in a large population in various dynamic systems has been one of the major problems. For instance, physicists often apply the traditional variational methods in Lagrangian and/or Hamiltonian mechanics to study interacting particle systems, which left a shortcoming of extremely high computational cost that made this microscopic approach almost mathematically intractable. To resolve this matter, a completely different macroscopic approach from statistical physics had been gradually developed, which eventually leads to the primitive notion of mean field theory. The novelty of this approach is that particles interact through a medium, namely the mean field term, which aggregates by action  and reaction on  other particles. Moreover, by passing the number of particles to the infinity in these macroscopic models, the mean field term become a functional of the density function which represents the whole population of particles. This leads to mathematical problems of much less computational complexity. 
 	
	\smallskip
	
From the economic perspective, due to the dramatic population growth and rapid urbanization, urgent needs of in-depth understanding of collective strategic interactive behavior of a huge group of decision makers is crucial in order to maintain a sustainable economic growth. Since the vector of fair prices is determined by both demand and supply, it is natural to utilize the aggregation effect from the players' states as a canonical candidate of mean-field term, and then we employ the mean-field models in place of the corresponding classical equilibrium models; moreover, as the decision makers control the evolution of a dynamic system, it is necessary to also incorporate the theory of stochastic differential games (SDGs) in these mean-field models. Over the past few decades, the theory of SDGs has been a major research topic in control theory and financial economics, especially in studying the continuous-time decision making problem between non-cooperative investors; in regard to the one-dimensional setting the theory of two person zero-sum games is quite well-developed via the notion of viscosity solutions, see for example Fleming and Souganidis (1989). Unfortunately, most interesting SDGs are $N$-player non-zero sum SDGs; see Bensoussan and Frehse~\cite{BF1,BF2} and Bensoussan et al.~\cite{BFV}, yet there are still relatively few results in the literature.
	
	\smallskip
			
As a macroscopic equilibrium model,  et al.~\cite{HCM1,HCM2} investigated stochastic differential game problems involving infinitely many players under the name ``Large Population Stochastic Dynamic Games''; and independently, Lasry and Lions~\cite{LL1,LL2,LL3} studied similar problems from the viewpoint of the mean-field theory in physics and termed ``Mean-Field Games (MFGs)''. As an organic combination of mean field theory and theory of stochastic differential games, MFGs provide a more realistic interpretation of individual dynamics at the microscopic level, so that each player  will be able to  optimize his prescribed objectives, yet with the mathematical tractability in a macroscopic framework. To be more precise, the general theory of MFGs has been built by combining various consistent assumptions on the following modeling aspects: (1) a continuum of players; (2) homogeneity in strategic performance of players; and (3) social interactions through the impact of mean field term. The first aspect is describing the approximation of a game model with a huge number of players by a continuum one yet with a sufficient mathematical tractability. The second aspect is assuming that all players obey the same set of rules of the interactive game, which provide guidance on their own behavior that potentially leads them to optimal decisions. Finally, due to the intrinsic complexity of the society in which the players participate in, the third aspect is explaining the fact that each player is so negligible and can only affect others marginally through his own infinitesimal contribution to the society. In a MFG, each player will base his decision making purely on his own criteria and certain summary statistics (that is, the mean field term) about the community; in other words, in explanation of their interactions, the pair of personal and mean-field characteristics of the whole population is already sufficient and exhaustive. Mathematically, each MFG will possess the following forward-backward structure: (1) a forward dynamic describes the individual strategic behavior; (2) a backward equation describes the evolution of individual optimal strategy, such as those in terms of the individual value function via the usual backward recursive techniques. For the detail of the derivation of this system of equations with forward-backward feature, one can consult from the works of Huang et al.~\cite{HCM2}, Lasry and Lions~\cite{LL1,LL2,LL3} and Bensoussan et al.~\cite{BFY}.
	
	\smallskip
	
In this article, we consider a class of MFG problems, in which there is a `significantly big' player playing together with a huge group of `small' players. The first work along this direction under a Linear Quadratic setting has been investigated by Huang \cite{H0}. In their following work~\cite{HN1}, the authors regard the mean field term, represented by the conditional expectation of the small agent, as exogenous to the whole control problem for both the big (the authors called it,`major') and small (minor) players. Nourian and Caines~\cite{NC} consider a similar problem under a generalized framework. However, the authors also consider the mean field term, which is represented by a conditional probability measure, as exogenous to the control problem for the major player. In contrast, we here consider the mean field term as endogenous for the big (we rephrase as `dominating' in order to emphasize our distinction from the previous works) player. That is to say, changes in the control of the big (dominating) player would directly affect and even essentially determine the mean field term. Our present setting appears to be natural in the economic literature related to `actual' governance, as the governor can often take up the initiative or key role on setting up rubrics and regulations to be followed by citizens. To avoid ambiguity, we here regard the `dominating' major player as a ``Dominating Player'', and all other minor players as ``Representative Agents'' throughout the whole paper. In our work, we assume that this dominating player can influence  both the mean field term and representative agents directly. We first discuss the necessary condition for the optimality under the most general setting in which both the state coefficients and the objective functions are sufficiently regular (e.g. differentiable); we then consider the Linear-Quadratic case by applying the results obtained in the general theory, which results in three adjoint equations. It is noted that Huang et al.~\cite{HN2} also considered the non-stationary case and obtained the intermediary result with only two adjoint equations, which represents a particular case of our present theory. Besides, concerning the related fixed point issue in any standard MFG problem in order to achieve the equilibrium strategy, we here only need to involve one single affine map, that simplifies much than that in \cite{HN2}, in which the authors need a couple of two similar mappings; apart from the simplicity of the sufficient condition provided here, it is also directly expressed in terms of the data (coefficients) of the underlying model. 

The paper is organized as follows: In Section \ref{gen}, we present the general theory of the Mean Field Games in the presence of a dominating Player, in which both the state coefficients and the objective functions are sufficiently regular. The necessary condition for  optimality and equilibrium is also provided there. Firstly, solving for the control problem of the representative agent, and then the equilibrium condition leads to a coupled Hamilton-Jacobi-Bellman and Fokker Planck equations. As the mean field term is endogenous to the dominating Player, in order to achieve an optimal control, he/she should take into account of the coupled equations when deciding his own controlling strategy. The related fixed point problem is described by six equations. In Section \ref{lqcase}, we study a special case with linear states together with linear quadratic objective functions. Due to natural coerciveness of the problem formulation, a necessary and sufficient condition for the optimality can be guaranteed. We write down both the stochastic maximum principle and the corresponding adjoint equations. In Section \ref{fxpt}, the corresponding fixed point problem is then tackled by considering the related Riccati equation, with which the equilibrium could be achieved. We then provide a `practical' sufficient condition, which only involves the data (coefficients) of the model without referring to any specific solution of any Riccati equations, for the existence of the equilibrium strategy. In Appendix, proof of the approximate Nash equilibrium for the general setting is also provided. 

\section{General Theory}\label{gen}
Consider a probability space $ (\Omega, \mathcal{F}, P) $, a fixed terminal time $T$ and two independent standard Brownian motion $W_0(t)$ and $W_1(t)$ taking values in $ \mathbb{R}^{d_0}$ and $\mathbb{R}^{d_1}$ respectively. Also consider two independent initial square integrable random variables $ \xi_0 \in \mathbb{R}^{n_0}$ and $\xi_1 \in \mathbb{R}^{n_1}$,  which are also assumed to be independent of both $W_0(t)$ and $W_1(t)$. 
Define the filtrations as follows, in which $ \mathcal{F}_t^0$  and  $ \mathcal{F}_t^1 $ are clearly independent to each other,
		\begin{equation*}
		\begin{split}
			\mathcal{F}_t^0& := \sigma(\xi_0, W_0(s), s \leq t) , \\
			\mathcal{F}_t^1 &:= \sigma(\xi_1, W_1(s), s \leq t) , \\
			\mathcal{G}_t &:= \mathcal{F}_t^0 \vee  \mathcal{F}_t^1 .
		\end{split}
		\end{equation*}
Let $\mathcal{P}_2(\mathbb{R}^{n_1})$ be the space of probability measures equipped with the $2^{nd}$ Wasserstein metric (for example, see \cite{villani}), $W_2(\cdot,\cdot)$ such that for any $\mu$ and $\nu$ in $\mathcal{P}_2(\mathbb{R}^{n_1})$,
\begin{equation*}
	W_2(\nu_1,\nu_2):=\inf_{\gamma\in\Gamma(\nu_1,\nu_2)}\bigg(\int_{\mathbb{R}^n\times\mathbb{R}^n}|x-y|^2d\gamma(x,y)\bigg)^{\frac{1}{2}},
\end{equation*}
where the infimum is taken over the family $\Gamma(\nu_1,\nu_2)$, the collection of all joint measures with respective marginals $\nu_1$ and $\nu_2$. Denote $d\lambda$ to be the Lebesgue measure on $\mathbb{R}^{n_1}$.
 
Denote $x_0(t)\in \mathbb{R}^{n_0}$ and $x_1(t)\in \mathbb{R}^{n_1}$ the state evolutions for the dominating player and a representative agent respectively whose dynamics are given by the following stochastic differential equations (SDEs),
		\begin{equation}
			\label{SDE}
			\begin{split}
			\left\{
			\begin{array}{rcl}
				dx_0 &=& g_0\Big(x_0(t),\mu(t),u_0(x_0(t),t)\Big)dt+\sigma_0\Big(x_0(t)\Big)dW_0(t) ,\\
				x_0(0) &=& \xi_0 .\\
				dx_1 &=&g_1\Big(x_1(t),x_0(t),\mu(t),u_1(x_1t)\Big)dt+\sigma_1\Big(x_1(t)\Big)dW_1(t) ,\\
				x_1(0) &=& \xi_1.
			\end{array}
			\right.
			\end{split}
		\end{equation}	

The functional coefficients are defined as follows:
\begin{equation}
\label{funcoef}
\left\{
\begin{array}{rl}
	&g_0:\mathbb{R}^{n_0} \times \mathcal{P}_2(\mathbb{R}^{n_1}) \times \mathbb{R}^{m_0} \rightarrow \mathbb{R}^{n_0} ,\\
	&g_1:\mathbb{R}^{n_1} \times \mathbb{R}^{n_0} \times \mathcal{P}_2(\mathbb{R}^{n_1}) \times \mathbb{R}^{m_1} \rightarrow \mathbb{R}^{n_1} ,\\
	&\sigma_0:\mathbb{R}^{n_0}\rightarrow \mathbb{R}^{n_0 \times d_0} ,\\
	&\sigma_1:\mathbb{R}^{n_1}\rightarrow \mathbb{R}^{n_1 \times d_1}. \\
\end{array}	
\right.
\end{equation}
The dominating player and the representative agents also possess the following objective functionals respectively:
	\begin{equation*}
	\begin{array}{rcl}
	J_0(u_0) &=& \mathbb{E}\left[\int_{0}^{T}f_0\Big(x_0(t),\mu(t),u_0(t)\Big)dt+h_0\Big(x_0(T),\mu(T)\Big)\right] ,\\
	J_1(u_1, x_0, \nu) &=& \mathbb{E}\left[\int_{0}^{T}f_1\Big(x_1(t),x_0(t),\mu(t),u_1(t)\Big)dt+h_1\Big(x_1(T),x_0(T),\mu(T)\Big)\right].
	\end{array}
	\end{equation*}
The functions are defined as follows:
	\begin{equation}
	\left\{
	\begin{array}{rl}
		&f_0:\mathbb{R}^{n_0} \times \mathcal{P}_2(\mathbb{R}^{n_1}) \times \mathbb{R}^{m_0} \rightarrow \mathbb{R},\\
		&f_1:\mathbb{R}^{n_1} \times \mathbb{R}^{n_0} \times \mathcal{P}_2(\mathbb{R}^{n_1}) \times \mathbb{R}^{m_1} \rightarrow \mathbb{R},\\
		&h_0:\mathbb{R}^{n_0} \times \mathcal{P}_2(\mathbb{R}^{n_1}) \rightarrow \mathbb{R} ,\\
		&h_1:\mathbb{R}^{n_1} \times \mathbb{R}^{n_0} \times  \mathcal{P}_2(\mathbb{R}^{n_1}) \rightarrow \mathbb{R}. \\
	\end{array}	
	\right.
	\end{equation}

Here $u_0\in \mathbb{R}^{m_0} $ and $u_1\in \mathbb{R}^{m_1} $ represent the respective controls of the dominating player and the representative agent. The controls $u_0$ and $u_1$ are respectively adapted to the filtrations $ \mathcal{F}_t^0$ and $ \mathcal{G}_t$. We further assume that the functional form (being a function of $(x_1(t),t)$) of $u_1$ is adapted to $ \mathcal{F}_t^0$ and uniformly Lipschitz in $x_1(t)$, even though its value evaluated at $x_1(t)$ would be adapted to $ \mathcal{G}_t$ instead. Loosely speaking the dominating player takes his own privilege of setting up the framework to be followed by the representative agent. We shall then define the classes of admissible controls for the dominating player and the representative agent by $\mathcal{A}_0$ and $\mathcal{A}_1$ respectively, where $\mathcal{A}_0$ (resp. $\mathcal{A}_1$) is a subset of $\mathcal{F}^0-$ (resp. $\mathcal{G}-$)progressively measurable process which are in $\mathcal{L}^2(\Omega\times[0,T];\mathbb{R}^{m_0})$ (resp. $\mathcal{L}^2(\Omega\times[0,T];\mathbb{R}^{m_1})$).

The mean field term, $\mu(t)\in\mathcal{P}_2(\mathbb{R}^{n_1})$, is the probability measure of the state of the representative agent at time $t$. Indeed, the dominating player sets rules for representative agent to take into account. One natural consideration is that the dominating player is incapable of tracing the state of each individual's evolution,  but only takes account of the overall performance of the community subject to the rules he set, that is his own flow of information, $\mathcal{F}^0_t$. By the same token, each agent cannot fully keep track of any other agents' states and they can only rely on the summarized information of the community provided by the dominating player. Thus it is justifiable to assume that the mean field term, $\mu(t)$, is adapted to $\mathcal{F}^0_t$.
The dominating player can directly influence  both the representative agent and the mean field term, thus we consider $\mu(t)$ as endogenous in the consideration of optimal behavior of $x_0(t)$ rather than as an exogenous variable commonly found in the literature such as that of \cite{H0,NC}.

For any probability measure $\mu\in\mathcal{P}_2(\mathbb{R}^{n_1})$, we write $M_2(\mu)=(\int_{\mathbb{R}^{n_1}} |x|^2 d\mu(x))^{\frac{1}{2}}$. We first give the following assumptions on the functional coefficients:

{\it (A.1) Lipschitz Continuity}\\
$g_0$, $\sigma_0$, $g_1$ and $\sigma_1$ are globally Lipschitz continuous in all arguments. In particular, there exists $K>0$, such that 
\begin{equation*}
\begin{array}{rcl}
|g_0(x_0,\mu,u_0)-g_0(x_0',\mu',u_0')|&\leq&  K\bigg(|x_0-x_0'|+W_2(\mu,\mu')+|u_0-u_0'|\bigg);\\
|\sigma_0(x_0)-\sigma_0(x_0')|&\leq& K|x_0-x_0'|.\\				
|g_1(x_1,x_0,\mu,u_1)-g_1(x_1',x_0',\mu',u_1')|&\leq&  K\bigg(|x_1-x_1'|+|x_0-x_0'|+W_2(\mu,\mu')+|u_1-u_1'|\bigg);\\
|\sigma_1(x_1)-\sigma_1(x_1')|&\leq& K|x_1-x_1'|.\\
\end{array}
\end{equation*}
{\it (A.2) Linear Growth}\\
$g_0$, $\sigma_0$, $g_1$ and $\sigma_1$ are of linear growth in all arguments. In particular, there exists $K>0$, such that
\begin{equation*}
\begin{split}
|g_0(x_0,\mu,u_0)|&\leq K(1+|x_0|+M_2(\mu)+|u_0|);\\
|\sigma_0(x_0)|&\leq K(1+|x_0|).\\
|g_1(x_1,x_0,\mu,u_1)|&\leq K(1+|x_0|+|x_1|+M_2(\mu)+|u_1|);\\
|\sigma_1(x_1)|&\leq K(1+|x_1|).\\
\end{split}
\end{equation*}
{\it (A.3) Quadratic Condition on the Cost Functional} (See (A.5) in Carmona and Delarue \cite{CAR_PROB}.)\\
There exists $K>0$, such that
\begin{equation}
\label{A4_0}
\begin{split}
|f_1(x_1,x_0,\mu,u_1)-f_1(x_1',x_0',\mu',u_1')|\leq &K\Big[1+|x_1|+|x_1'|+|x_0|+|x_0'|+|M_2(\mu)|+|M_2(\mu')|+|u_1|+|u_1'|\Big]\\
&\qquad\cdot\Big[|x_1-x_1'|+|x_0-x_0'|+W_2(\mu,\mu')+|u_1-u_1'|\Big];\\
|h_1(x_1,x_0,\mu)-h_1(x_1',x_0',\mu')|\leq & K\Big[1+|x_1|+|x_1'|+|x_0|+|x_0'|+|M_2(\mu)|+|M_2(\mu')|\Big]\\
&\qquad\cdot\Big[|x_1-x_1'|+|x_0-x_0'|+W_2(\mu,\mu')\Big].
\end{split}
\end{equation}	
Under the assumptions {\it A.1-A.3}, we show in the Appendix that if we have the mean field term coincides with the probability measure of $x_1(t)$ conditioning on $\mathcal{F}^0_t$, then the optimization problem for the representative agent constitutes to a Mean Field Game. In general, it is more convenient to compare two probability measures if they possess density functions on $\mathbb{R}^{n_1}$. We define the second order operator $A_1$ and its adjoint $A_1^*$ by
\begin{equation*}
\begin{split}
A_1 \varphi (x,t) &=-\text{tr}\Big(a_1(x)D^2\varphi(x,t)\Big) , \\
A_1^{*}\varphi(x,t)&=-\sum_{i,j=1}^{n_1}\frac{\partial^{2}}{\partial{x_{i}}\partial{x_{j}}}\Big(a_1^{ij}(x)\varphi(x,t)\Big),
\end{split}
\end{equation*}
where $a_1(x)=\frac{1}{2}\sigma_1(x)\sigma_1(x)^*$ is a positive definite matrix.
Let $x_1=x_1^{u_1}$ be the solution of the SDE for the representative agent with respect to control $u_1$. For any test function $f$, by It\^o's lemma,
\begin{equation}
	\mathbb{E}^{\mathcal{F}^0_t}[f(x_1^{u_1}(t))]=\mathbb{E}^{\mathcal{F}^0_t}\bigg[f(\xi_1)+\displaystyle\int_0^t \big(\partial_t+g\cdot D+A_1\big)f ds\bigg],
\end{equation}
The conditional density function $p^{u_1}(\cdot,t)$ of $x_1^{u_1}(t)$ (if exists), i.e. $\mathbb{E}^{\mathcal{F}^0_t}[f(x_1^{u_1}(t))]=\int_{\mathbb{R}^{n_1}}f(x)p^{u_1}(x,t)dx$, would be given by the Fokker-Planck (FP) equation
\begin{equation}\label{FPEQ}
\begin{split}
	\left\{
	\begin{array}{rl}
		\dfrac{\partial p^{u_1}}{\partial t}&=-A_1^{*}p^{u_1}(x,t)-\text{div}\Big(g_1\Big(x,x_{0}(t),\mu(t),u_1(x,t)\Big)p^{u_1}(x,t)\Big), \\
		p^{u_1}(x,0) &=\omega(x),
	\end{array}
	\right.
\end{split}
\end{equation}
where $\omega(x)$ is the initial density function of $\xi_1$. We will justify the existence and regularities of the conditional density function $p^{u_1}$ later. We first assume $p^{u_1}(\cdot,t)\in\mathcal{L}^2(\mathbb{R}^{n_1})$ and $p^{u_1}(\cdot,t)d\lambda\in\mathcal{P}_2(\mathbb{R}^{n_1})$. For any density $m$, we may write $m d\lambda=m$ if no ambiguity arises. We impose further assumptions on the functional coefficients:

{\it (A.4)}\\
	$g_0, f_0, h_0, g_1, f_1$ and $h_1$ are continuously differentiable in (if the argument exist) $x_0\in\mathbb{R}^{n_0}$, $x_1\in\mathbb{R}^{n_1}$ $u_0\in\mathbb{R}^{m_0}$, $u_1\in\mathbb{R}^{m_1}$ with bounded derivatives. We will denote, for example, the derivative of $g_0$ with respect to $x_0$ by $g_{0,x_0}$. They are also G\^ateaux differentiable in $\mu=md\lambda\in\mathcal{P}_2(\mathbb{R}^{n_1})$, for example, for $m\in\mathcal{L}^2(\mathbb{R}^{n_1})$,
		\begin{equation*}
			\at{\dfrac{d}{d\theta}}{\theta=0}g_0(x_i,(m+\theta \tilde m)d\lambda,u_i)=\int_{\mathbb{R}^n}\frac{\partial g_0}{\partial m}(x_i,md\lambda,u_i)(\xi)\tilde m(\xi) d\xi,
		\end{equation*}
		for some $\frac{\partial g_0}{\partial m}(x_i,md\lambda,u_i)\in\mathcal{L}^2(\mathbb{R}^{n_1})$.
		
{\it (A.5)}\\
	$\sigma_0$ (resp. $\sigma_1$) is twice continuously differentiable in $x_0\in\mathbb{R}^{n_0}$ (resp. $x_1\in\mathbb{R}^{n_1}$) with bounded first order and second order derivative. 
\begin{rk}
	With the regularities on the coefficients, if we have the initial density $\omega(x)\in\mathcal{L}^2\cap\mathcal{L}^{\infty}(\mathbb{R}^{n_1})$, then the FP equation (\ref{FPEQ}) admits a unique solution $p^{u_1}\in \mathcal{L}^{\infty}([0,T],\mathcal{L}^2\cap\mathcal{L}^{\infty}(\mathbb{R}^{n_1}))$. See Proposition $4$ and $5$ in Le Bris and Lions \cite{BL} for details.
\end{rk}
Define then a pair of mutually dependent control problems for the dominating player and the representative agent as below:
\begin{pblm}{\bf Control of Representative Agent}\label{P1} \\
	Given the process $x_0$ and an exogenous probability measure-valued process $\nu$ (adapted to $\mathcal{F}^0_t$), find a control $u_1 \in \mathcal{A}_1$ which minimizes the cost functional
	\begin{equation}
	\label{J_1}
		J_1(u_1, x_0, \nu) := \mathbb{E}\left[\int_{0}^{T}f_1\Big(x_1(t),x_0(t),\nu(t),u_1(t)\Big)dt+h_1\Big(x_1(T),x_0(T),\nu(T)\Big)\right].
	\end{equation}
\end{pblm}
\begin{pblm}{\bf Equilibrium Condition}\label{P2} \\
	Given an exogenous probability measure-valued process $\nu$, let $\mathcal{M}(\nu)(t)$
	be the measure induced by the corresponding optimal $x_1(t)$ found in Problem \ref{P1} conditioning on $\mathcal{F}_t^0$. Find the probability measure-valued process $\mu$ such that the fixed point property is satisfied: $\mathcal{M}(\mu)(\cdot)=\mu(\cdot)$.\\
\end{pblm}
\begin{pblm}{\bf Control of the Dominating Player}\label{P3} \\
	Find a control $u_0 \in \mathcal{A}_0$ which minimizes the cost functional \\
	\begin{equation}
		J_0(u_0) := \mathbb{E}\left[\int_{0}^{T}f_0\Big(x_0(t),\mu(t),u_0(t)\Big)dt+h_0\Big(x_0(T),\mu(T)\Big)\right] ,
	\end{equation} 
	where $\mu$ is the solution given in Problem \ref{P2}.
\end{pblm}

\begin{rk}
The setting of our problem is different from those mean field related problems with a major player (not a dominating player) commonly found in the literature, such as that in \cite{H0,NC}. For example, in \cite{H0}, the corresponding objective functions for the major player and i-th minor player are respectively
\begin{equation*}
\begin{array}{rcl}
	J_0(u_0,z)&=&\mathbb{E}\int_0^T \left\{| x_0-H_0 z-\eta |^2_{Q_0}+u_0^*R_0 u_0 \right\}dt, \\
	J_i(u_i,z)&=&\mathbb{E}\int_0^T \left\{| x_i-H x_0-\hat{H}_0 z-\eta |^2_{Q}+u^*R u \right\}dt, \\
\end{array}
\end{equation*}
where the mean field term $z$ is exogenous to both control optimization problems for $J_0$ and $J_i$. Instead, we here consider the mean field term $\nu$, as established in Problem \ref{P2}, as endogenous for the dominating player in Problem \ref{P3}. In particular, changes in control $u_0$ would affect and even completely determine the mean field term $\nu$ accordingly. Our setting appears to be natural in the economic context related to governance, as the governor can sometimes take the initiative to set-up rubrics to be obeyed and followed by citizens; this latter notion is  covered in \cite{ORIGINCRISIS}. 
\end{rk}

We first establish the necessary condition of optimality for the representative agent (Problem \ref{P1}) by the adjoint equation approach. After resolving Problem \ref{P1}, we solve for the fixed point in Problem \ref{P2}. Recall that $x_0(t)$, the functional form of $u_1$ (conditioning on $\mathcal{F}_0^t$, $u_1$ is a function of $(x_1(t),t)$) and now together with the input measure-valued process $\nu$ are all adapted to $\mathcal{F}_t^0$. We can then rewrite the cost functional (\ref{J_1}) for the representative agent as

\begin{equation*}
\begin{split}
	J_1(u_1, x_0, \nu) =&\mathbb{E}\bigg[\int_{0}^{T}\mathbb{E}^{\mathcal{F}_t^0}f_1\Big(x_1(t),x_0(t),\nu(t),u_1(x_1(t),t)\Big)dt+\mathbb{E}^{\mathcal{F}_T^0}h_1\Big(x_1(T),x_0(T),\nu(T)\Big)\bigg] \\
	=&\mathbb{E}\bigg[\int_{0}^{T}\int_{\mathbb{R}^{n_1}}p_{u_1}(x,t)f_1\Big(x,x_{0}(t),\nu(t),u_1(x,t)\Big)dxdt +  \int_{\mathbb{R}^n}p_{u_1}(x,T)h_1\Big(x,x_{0}(T),\nu(T)\Big)dx\bigg].
\end{split}
\end{equation*}

\begin{lem} {\bf (Necessary condition for Problem \ref{P1})} \\
	Given $x_0$ and $\nu$ as in Problem \ref{P1}, the control $\hat{u}_1 \in \mathcal{A}_1$ is optimal only if it satisfies the following SHJB:\\ 
	\begin{equation}
	\left\{
	\begin{array}{rl}
		\label{SHJB}
		-\partial_{t}\Psi=& \Big(H_1(x,x_0(t),\nu(t),D\Psi(x,t))-A_1\Psi(x,t)\Big)dt  -K_{\Psi}(x,t)dW_0(t), \\
		\Psi(x,T)  =&h_1\Big(x,x_{0}(T),\nu(T)\Big),\\
	\end{array}
	\right.
	\end{equation}
	where
	\begin{equation*}
	\begin{split}
		H_1(x,x_0,\nu,q)&=\inf_{u_1} L(x,x_0,\nu,u_1,q),\\
		L(x,x_0,\nu,u_1,q)&=f_1(x,x_0,\nu,u_1)+qg_1(x,x_0,\nu,u_1).\\
	\end{split}
	\end{equation*}
	and the infimum is uniquely attained at $\hat u_1$, i.e. $H_1(x,x_0,\nu,q)=L(x,x_0,\nu,\hat u_1,q)$.
\end{lem}
\begin{rk}
	As in the work in Carmona et al. \cite{CAR_PROB}, one convenient set of assumptions on $g_1$,$f_1$ and $h_1$ which ensures the unique existence of the minimizer, $\hat u_1(x,x_0,\nu,q)={\arg\min}_u L(x,x_0,\nu,u,q)$, is the affine and convexity assumption. See Lemma $2.1$ in \cite{CAR_PROB} for more details. In particular, for $x_1\in\mathbb{R}^{n_1}; x_0\in\mathbb{R}^{n_0}; \mu\in \mathcal{P}_2(\mathbb{R}^{n_1})$ and $ u_1,u_1'\in\mathbb{R}^{m_1}$, there exists $K>0$ such that
	\begin{enumerate}
		\item{$g_1(x_1,x_0,\mu,u_1)=g_{1,1}(x_1,x_0,\mu)+g_{1,2}\cdot u_1$,}
		\item{$f_1(x_1',x_0,\mu,u_1') \geq f_1(x_1,x_0,\mu,u_1)+f^*_{1,x_1}(x_1,x_0,\mu,u_1)(x_1'-x_1)+f^*_{1,u_1}(x_1,x_0,\mu,u_1)(u_1'-u_1)+K|u_1'-u_1|^2$.}
	\end{enumerate}
	Moreover, the minimizer $(x,q)\mapsto \hat{u}_1(x,x_0,\mu,q)$ is Lipschitz continuous, uniformly in $(x_0,\mu)$. Similar conditions for $g_0$, $f_0$ and $h_0$ can be assumed to guarantee a unique minimizer for the Lagrangian of the control problem for the dominating player in Lemma \ref{LEMP3}.
\end{rk}
\begin{proof}
To express a necessary condition for optimality, we adopt the stochastic maximum principle. In particular, for any $\tilde u_1\in\mathcal{A}_1$,
\begin{equation}\label{1storder}
\begin{array}{rcl}
	0&=&\at{\dfrac{d}{d\theta}}{\theta=0}J_1(\hat u_1+\theta\tilde{u}_1,x_{0},\nu)\\
	&=&\at{\dfrac{d}{d\theta}}{\theta=0}\mathbb{E}\bigg[\displaystyle\int_{0}^{T}\int_{\mathbb{R}^{n_1}}p_{\hat u_1+\theta\tilde{u}_1}(x,t)f_1\Big(x,x_{0}(t),\nu(t),\hat u_1(x,t)+\theta\tilde{u}_1(x,t)\Big)dxdt \\
	&&\qquad\qquad\qquad  +\displaystyle\int_{\mathbb{R}^n}p_{\hat u_1+\theta\tilde{u}_1}(x,T)h_1\Big(x,x_{0}(T),\nu(T)\Big)dx\bigg]\\
	&=&\mathbb{E}\bigg[\displaystyle\int_{0}^{T}\int_{\mathbb{R}^{n_1}}\tilde{p}(x,t)f_1\Big(x,x_{0}(t),\nu(t),\hat u_1(x,t)\Big)+p_{\hat{u}_1}(x,t)f_{1,u_1}\Big(x,x_{0}(t),\nu(t),\hat u_1(x,t)\Big)\tilde{u}_1(x,t)dxdt \\
	&&\qquad\qquad\qquad  +\displaystyle\int_{\mathbb{R}^n}\tilde{p}(x,T)h_1\Big(x,x_{0}(T),\nu(T)\Big)dx\bigg],
\end{array}
\end{equation}
where $\tilde{p}=\at{\frac{d}{d\theta}}{\theta=0}p_{\hat{u}_1+\theta \tilde{u}_1}$. By taking derivative with respect to $\theta$ in the FP equation (\ref{FPEQ}), we have
\begin{equation*}
\begin{split}
	\left\{
	\begin{array}{rl}
		\dfrac{\partial \tilde{p}}{\partial t}&=-A_1^{*}\tilde{p}(x,t)-\text{div}\Big(\tilde{u}_1(x,t)g_{1,u_1}\Big(x,x_{0}(t),\mu(t),\hat{u}_1(x,t)\Big)p_{\hat{u}_1}(x,t)\Big)\\
		&\qquad\qquad-\text{div}\Big(g_1\Big(x,x_{0}(t),\mu(t),\hat{u}_1(x,t)\Big)\tilde{p}(x,t)\Big),\\
		\tilde{p}(x,0) &=0.
	\end{array}
	\right.
\end{split}
\end{equation*}
As an adjoint process, we consider the backward stochastic differential equation
\begin{equation*}
\left\{
\begin{array}{rcl}
	-\partial_{t}\Psi &=& \Big\{f_1\Big(x,x_{0}(t),\nu(t),u_1(x,t)\Big)+D\Psi(x,t)g_1\Big(x,x_{0}(t),\nu(t),u_1(x,t)\Big)-A_1\Psi(x,t)\Big\}dt\\
	&&\qquad-K_{\Psi}(x,t)dW_0(t) ,\\
	\Psi(x,T)  &=&h_1\Big(x,x_{0}(T),\nu(T)\Big) .
\end{array}
\right.
\end{equation*}
We consider the inner product
\begin{equation*}
\begin{array}{rcl}
	&&d\displaystyle\int_{\mathbb{R}^{n_1}} \tilde{p}(x,t) \Psi(x,t)dx\\
	&=&\displaystyle\int_{\mathbb{R}^{n_1}}\Big\{-A_1^{*}\tilde{p}(x,t)-\text{div}\Big(\tilde{u}_1(x,t)g_{1,u_1}\Big(x,x_{0}(t),\mu(t),\hat{u}_1(x,t)\Big)p_{\hat{u}_1}(x,t)\Big)\\
	&&\qquad\qquad-\text{div}\Big(g_1\Big(x,x_{0}(t),\mu(t),\hat{u}_1(x,t)\Big)\tilde{p}(x,t)\Big)\Big\}\Psi(x,t)dxdt\\
	&&-\displaystyle\int_{\mathbb{R}^{n_1}}\tilde{p}(x,t)\Big\{f_1\Big(x,x_{0}(t),\nu(t),u_1(x,t)\Big)+D\Psi(x,t)g_1\Big(x,x_{0}(t),\nu(t),u_1(x,t)\Big)-A_1\Psi(x,t)\Big\}dxdt\\
	&&+\displaystyle\int_{\mathbb{R}^{n_1}}\tilde{p}(x,t)K_{\Psi}(x,t)dxdW_0(t)\\
	&=&\displaystyle\int_{\mathbb{R}^{n_1}}\Big(\tilde{u}_1(x,t)g_{1,u_1}\Big(x,x_{0}(t),\mu(t),\hat{u}_1(x,t)\Big)p_{\hat{u}_1}(x,t)\Big)D\Psi(x,t)dxdt\\
	&&-\displaystyle\int_{\mathbb{R}^{n_1}}\tilde{p}(x,t)f_1\Big(x,x_{0}(t),\nu(t),u_1(x,t)\Big)dxdt+\displaystyle\int_{\mathbb{R}^{n_1}}\tilde{p}(x,t)K_{\Psi}(x,t)dxdW_0(t).\\	
\end{array}
\end{equation*}
Taking integration on $[0,T]$ and expectation on both sides yields
\begin{equation*}
\begin{array}{rcl}
	&&\mathbb{E}\bigg[\displaystyle\int_{\mathbb{R}^n}\tilde{p}(x,T)h_1\Big(x,x_{0}(T),\nu(T)\Big)dx\bigg]\\
	&=&\mathbb{E}\bigg[\displaystyle\int_0^T\displaystyle\int_{\mathbb{R}^{n_1}}\Big(\tilde{u}_1(x,t)g_{1,u_1}\Big(x,x_{0}(t),\mu(t),\hat{u}_1(x,t)\Big)p_{\hat{u}_1}(x,t)\Big)D\Psi(x,t)dxdt\\
	&&\qquad-\displaystyle\int_0^T\displaystyle\int_{\mathbb{R}^{n_1}}\tilde{p}(x,t)f_1\Big(x,x_{0}(t),\nu(t),u_1(x,t)\Big)dxdt\bigg].
\end{array}
\end{equation*}
Together with the first order condition (\ref{1storder}), we have
\begin{equation*}
	0=\mathbb{E}\bigg[\displaystyle\int_0^T\displaystyle\int_{\mathbb{R}^{n_1}}\tilde{u}_1(x,t)\Big[g_{1,u_1}\Big(x,x_{0}(t),\mu(t),\hat{u}_1(x,t)\Big)D\Psi(x,t)+f_{1,u_1}\Big(x,x_{0}(t),\nu(t),\hat u_1(x,t)\Big)\Big]p_{\hat{u}_1}(x,t)dxdt\bigg].
\end{equation*}
Recall that the $p_{\hat{u}_1}(\cdot,t)$ is a conditional probability density function and hence non-negative, and $\tilde{u}_1$ is an arbitrary Markovian control, we have $\hat{u}_1$ is optimal only if 
\begin{equation*}
	g_{1,u_1}\Big(x,x_{0}(t),\mu(t),\hat{u}_1(x,t)\Big)D\Psi(x,t)+f_{1,u_1}\Big(x,x_{0}(t),\nu(t),\hat u_1(x,t)\Big)=0, a.e. (x,t).
\end{equation*}
With the definition of $L$ in the theorem, the condition becomes
\begin{equation*}
	L_{u_1}(x,x_0(t),\nu(t),\hat{u}_1(x,t),D\Psi(x,t))=0, a.e.(x,t),
\end{equation*}
which provides a necessary condition for the minimization problem. As the minimizer is assumed to be attained at $\hat{u}_1$, which depends on $x$, $x_0$, $\nu$, and $D\Psi$, we arrive for the SHJB Equation.
\qed \end{proof}

Replace the arbitrary measure $\nu$ by the mean field measure $\mu$. Equating $\mu:=m_{x_0}d\lambda$ with  $p_{\hat u_1}d\lambda$, the measure of the optimal state of the representative agent conditioning on $\mathcal{F}^0_t$; the couple (\ref{FPEQ}) and (\ref{SHJB}) give the following corollary.
\begin{cor}{\bf (Necessary condition for Problems \ref{P1} and \ref{P2})}\label{P1P2} \\
	The control for the representative agent is optimal and the equilibrium condition holds only if the SHJB-FP coupled equations are satisfied\\
	\begin{equation}
	\label{HJBFP}
	\left\{
	\begin{array}{rcl}
		-\partial_{t}\Psi&=& \Big(H_1\Big(x,x_0(t),m_{x_0}(x,t),D\Psi(x,t)\Big)-A_1\Psi(x,t)\Big)dt  -K_{\Psi}(x,t)dW_0(t), \\
		\Psi(x,T)  &=&h_1(x,x_{0}(T),m_{x_0}(x,T)).\\
		\dfrac{\partial m_{x_0}}{\partial t}&=&-A_1^{*}m_{x_0}(x,t)-{\normalfont \text{div}}\Big(G_1\Big(x,x_0(t),m_{x_0}(x,t),D\Psi(x,t)\Big)m_{x_0}(x,t)\Big),\\
		m_{x_0}(x,0) &=&\omega(x),\\
	\end{array}
	\right.
	\end{equation}
	where $G_1(x,x_0,m,q)=g_1(x,x_0,m,\hat{u}_1(x,x_0,m,q))$.
\end{cor}
The SHJB-FP coupled equations (\ref{HJBFP}) allow us to obtain the control of the representative agent in terms of a given trajectory of the dominating player $x_0$ while the equilibrium condition also holds. 

We then turn to the optimal problem for the dominating player. As $m_{x_0}$ is not external to the dominating player, the dominating player has to consider both its own dynamic evolution and (\ref{HJBFP}).
\begin{lem}{\bf Necessary condition for Problem \ref{P3}}\\
\label{LEMP3}
	The control for the dominating player $\hat u_0$ is optimal only if
	\begin{equation*}
	\begin{split}
		f_0(x_0,m_{x_0},\hat u_0)+p \cdot g_0(x_0,m_{x_0},\hat u_0)&= \inf_{u_0} \Big\{f_0(x_0,m_{x_0},u_0)+p \cdot g_0(x_0,m_{x_0},u_0)\Big\} \\
		:&=H_0(x_0,m_{x_0},p)\\
	\end{split}
	\end{equation*}
	where $p(t)$ satisfies the following adjoint processes
	\begin{equation*}
	\begin{array}{rcl}
		-dp&=&\bigg[g_{0,x_{0}}^*(x_{0}(t),m_{x_0}(x,t),\hat u_{0}(t))p(t)+f_{0,x_{0}}(x_{0}(t),m_{x_0}(x,t),\hat u_{0}(t))\\
		&&+\displaystyle \int G_{1,x_{0}}(x,x_{0}(t),m_{x_0}(x,t),D\Psi{(x,t)})D{q}(x,t)m_{x_0}(x,t)dx\\
		&&+\displaystyle \int r(x,t)H_{1,x_{0}}(x,x_{0}(t),m_{x_0}(x,t),D\Psi{(x,t)})dx\bigg]dt\\
		&&-\sum_{l=1}^{d_0}K_p^l(t)dW_0^l(t) +\sum_{l=1}^{d_0}\sigma_{0,x_{0}}^{l*}(x_0(t))K_p^l(t)dt, \\
		p(T)&=&h_{0,x_0}(x_0(T),m_{x_0}(x,T))+\displaystyle 	\int{r}(x,T)h_{1,x_0}(x,x_0(T),m_{x_0}(x,T))dx; \\
	\end{array}
	\end{equation*}
	\begin{equation*}
	\begin{array}{rcl}
		-\partial_{t}{q}&=&\bigg[-A_1{q}(x,t)+p(t)\dfrac{\partial g_{0}}{\partial m}(x_{0}(t),m_{x_0}(\xi,t),\hat u_{0}(t))(x)\\
		&&+D{q}(x,t)G_1(x,x_{0}(t),m_{x_0}(x,t),D\Psi{(x,t)})\\
		&&+\displaystyle \int D{q}(\xi,t)\dfrac{\partial G_1}{\partial m}(\xi,x_{0}(t),m_{x_0}(\xi,t),D\Psi{(\xi,t)})(x)m_{x_0}(\xi,t)d\xi\\
		&&+\displaystyle \int{r}(\xi,t)\dfrac{\partial H_1}{\partial m}(\xi,x_{0}(t),m_{x_0}(\xi,t),D\Psi{(\xi,t)})(x)d\xi\\
		&&+\dfrac{\partial f_{0}}{\partial m}(x_{0}(t),m_{x_0}(x,t),\hat u_{0}(t))(x)\bigg]dt-K_q(x,t)dW_0(t),\\
		{q}(x,T)&=&\dfrac{\partial h_{0}}{\partial m}(x_{0}(T),m_{x_0}(\xi,T))(x)+\displaystyle \int {r}(\xi,T)\dfrac{\partial h_1}{\partial m}(\xi,x_0(T),m_{x_0}(\xi,T))(x)d\xi; 
	\end{array}
	\end{equation*}
	\begin{equation*}
	\begin{array}{rcl}
		\dfrac{\partial{r}}{\partial t}&=&-A_1^{*}{r}(x,t)-\text{\rm div}\bigg[{r}(x,t)H_{1,q}(x,x_{0}(t),m_{x_0}(x,t),D\Psi{(x,t)})\\
		&&+G_{1,q}(x,x_{0}(t),m_{x_0}(x,t),D\Psi{(x,t)})D{q}(x,t)m_{x_0}(x,t)\bigg],\\
		{r}(x,0)&=&0.
	\end{array}
	\end{equation*}
\end{lem}
\begin{proof}
Again we consider the G\^{a}teaux derivative 
\begin{equation}
	\begin{split}
	\label{gatJ0}
			0=&\at{\dfrac{d}{d\theta}}{\theta=0}J_0(\hat u_0+\theta\tilde{u}_0) \\
			=&\mathbb{E}\left\{\int_{0}^{T}[f_{0,x_{0}}(x_{0}(t),m_{x_0}(x,t),\hat u_{0}(t))\tilde{x}_{0}(t)+\int\dfrac{\partial f_{0}}{\partial m}(x_{0}(t),m_{x_0}(x,t),\hat u_{0}(t))(\xi)\tilde{m}_{x_0}(\xi,t)d\xi\right.\\
			&+f_{0,u_{0}}(x_{0}(t),m_{x_0}(x,t),\hat u_{0}(t))\tilde{u}_{0}(t)]dt\\
			&+\left. h_{0,x_{0}}(x_{0}(T),m_{x_0}(x,T))\tilde{x}_{0}(T)+\int\dfrac{\partial h_{0}}{\partial m}(x_{0}(T),m_{x_0}(x,T))(\xi)\tilde{m}_{x_0}(\xi,T)d\xi\right\},
		\end{split}
		\end{equation}
		where $\tilde{x}_0=\at{\frac{d}{d\theta}}{\theta=0}x_{0}(\hat u_0+\theta\tilde{u}_0)$; 
		$\tilde{m}_{x_0}=\at{\frac{d}{d\theta}}{\theta=0}m_{x_0}(\hat u_0+\theta\tilde{u}_0)$;
		$\tilde{\Psi}=\at{\frac{d}{d\theta}}{\theta=0}\Psi(\hat u_0+\theta\tilde{u}_0)$,
		satisfy
	\begin{equation*}
	\begin{array}{rcl}
		d\tilde{x}_{0}&=&\bigg[g_{0,x_{0}}(x_{0}(t),m_{x_0}(x,t),\hat u_{0}(t))\tilde{x}_{0}(t)+\displaystyle \int\dfrac{\partial g_{0}}{\partial m}(x_{0}(t),m_{x_0}(x,t),\hat u_{0}(t))(\xi)\tilde{m}_{x_0}(\xi,t)d\xi\\
		&&+g_{0, u_{0}}(x_{0}(t),m_{x_0}(x,t),\hat u_{0}(t))\tilde{u}_{0}(t)\bigg]dt+\sum_{l=1}^{d_0}\sigma_{0,x_{0}}^l (x_{0}(t))\tilde{x}_{0}(t)dW_0^l(t),\\
		\tilde{x}_{0}(0)&=&0; \\
	\end{array}
	\end{equation*}
	\begin{equation*}
	\begin{array}{rcl}
		\dfrac{\partial \tilde{m}_{x_0}}{\partial t}&=&-A_1^*\tilde{m}_{x_0}(x,t)-\text{\rm div}\bigg\{\Big[G_{1,x_{0}}(x,x_{0}(t),m_{x_0}(x,t),D\Psi(x,t))\tilde{x}_{0}(t)\\
		&&+\displaystyle \int\dfrac{\partial G_1}{\partial m}(x,x_{0}(t),m_{x_0}(x,t),D\Psi(x,t))(\xi)\tilde{m}_{x_0}(\xi,t)d\xi \\
		&&+G_{1,q}(x,x_{0}(t),m_{x_0}(x,t),D\Psi(x,t))D\tilde{\Psi}(x,t)\Big]m_{x_0}(x,t)\\
		&&+G_1(x,x_{0}(t),m_{x_0}(x,t),D\Psi(x,t))\tilde{m}_{x_0}(x,t)\bigg\},\\
		\tilde{m}_{x_0}(x,0)&=&0;\\
	\end{array}
	\end{equation*}
	\begin{equation*}
	\begin{array}{rcl}
		-\partial_{t}\tilde{\Psi}&=&\bigg[H_{1,x_0}(x,x_{0}(t),m_{x_0}(x,t),D\Psi(x,t))\tilde{x}_{0}(t)\\
		&&+\displaystyle \int\dfrac{\partial H_1}{\partial m}(x,x_{0}(t),m_{x_0}(x,t),D\Psi(x,t))(\xi)\tilde{m}_{x_0}(\xi,t)d\xi\\
		&&+H_{1,q}(x,x_{0}(t),m_{x_0}(x,t),D\Psi(x,t))D\tilde{\Psi}(x,t)-A_1\tilde{\Psi}(x,t)\bigg]dt-K_{\tilde{\Psi}}(x,t)dW_0(t)\nonumber ,\\
		\tilde{\Psi}(x,T)&=&h_{1,x_0}(x,x_0(T),m_{x_0}(x,T))\tilde{x}_0(T)+\displaystyle \int \dfrac{\partial h_1}{\partial m}(x,x_0(T),m_{x_0}(x,T))(\xi) \tilde{m}_{x_0}(\xi,T)d\xi. \\
	\end{array}
	\end{equation*}		
	Introduce the adjoint processes $p(t)$, ${q}(x,t)$ and ${r}(x,t)$ as stated in the lemma statement and consider the following differentials\\
	\begin{equation*}
	\begin{array}{rcl}
		&&d(p^* \tilde x_0)\\
		&=&-\tilde x_0^*(t)\bigg[f_{0,x_{0}}(x_{0}(t),m_{x_0}(x,t),\hat u_{0}(t))+\displaystyle \int{r}(x,t)H_{1,x_{0}}(x,x_{0}(t),m_{x_0}(x,t),D\Psi(x,t))dx\\
		&&+\displaystyle \int G_{1,x_{0}}(x,x_{0}(t),m_{x_0}(x,t),D\Psi{(x,t)})D{q}(x,t)m_{x_0}(x,t)dx\bigg]dt\\
		&&+p^*(t)\bigg[\displaystyle \int\dfrac{\partial g_{0}}{\partial m}(x_{0}(t),m_{x_0}(x,t),\hat u_{0}(t))(\xi)\tilde{m}_{x_0}(\xi,t)d\xi\\
		&&+g_{0,u_0}(x_{0}(t),m_{x_0}(x,t),\hat u_{0}(t))\tilde{u}_{0}(t)\bigg]dt+\left\{\dots\right\}dW_0(t).\\
	\end{array}
	\end{equation*}
	\begin{equation*}
	\begin{array}{rcl}
		&&d\displaystyle \int {q}(x,t)\tilde{m}_{x_0}(x,t)dx\\
		&=&\displaystyle \int {q}(x,t)\bigg[-\text{\rm div}\Big[[G_{1,x_{0}}(x,x_{0}(t),m_{x_0}(x,t),D\Psi(x,t))\tilde{x}_0(t)\\
		&&+G_{1,q}(x,x_{0}(t),m_{x_0}(x,t),D\Psi(x,t))D\tilde{\Psi}]m_{x_0}(x,t)\Big]\bigg]dxdt\\
		&&-\displaystyle \int\Big[p(t)\dfrac{\partial g_{0}}{\partial m}(x_{0}(t),m_{x_0}(\xi,t),\hat u_{0}(t))(x)+\displaystyle \int{r}(\xi,t)\dfrac{\partial H_1}{\partial m}(\xi,x_{0}(t),m_{x_0}(\xi,t),D\Psi{(\xi,t)})(x)d\xi\\
		&&+\dfrac{\partial f_{0}}{\partial m}(x_{0}(t),m_{x_0}(x,t),\hat u_{0}(t))(x)\Big]\tilde m_{x_0}(x,t)dxdt+\left\{\dots\right\}dW_0(t).\\
	\end{array}
	\end{equation*}
	\begin{equation*}
	\begin{array}{rcl}
		&&d\displaystyle \int {r}(x,t)\tilde{\Psi}(x,t)dx\\	
		&=&\displaystyle \int\bigg[-\text{\rm div}\Big[G_{1,q}(x,x_{0}(t),m_{x_0}(x,t),D\Psi{(x,t)})D{q}(x,t)m_{x_0}(x,t)\Big]\bigg]\tilde{\Psi}(x,t)dxdt\\
		&&-\displaystyle \int r(x,t)\bigg[H_{1,x_{0}}(x,x_{0}(t),m_{x_0}(x,t),D\Psi(x,t))\tilde{x}_{0}(t)\\
		&&+\displaystyle \int\dfrac{\partial H_1}{\partial m}(x,x_{0}(t),m_{x_0}(x,t),D\Psi(x,t))(\xi)\tilde{m}_{x_0}(\xi,t)d\xi\bigg]dxdt+\left\{\dots\right\}dW_0(t).\\	
	\end{array}
	\end{equation*}
	Using the results above, we have
	\begin{equation*}
	\begin{array}{rcl}
		&&d\displaystyle \Big\{p^*\tilde{x}_0+\int {q}(x,t)\tilde{m}_{x_0}(x,t)dx-\int {r}(x,t)\tilde{\Psi}(x,t)dx\Big\}\\
		&=&\bigg[\displaystyle-\tilde x_0^*(t)f_{0,x_{0}}(x_{0}(t),m_{x_0}(x,t),\hat u_{0}(t))+p^*(t)g_{0,u_0}(x_{0}(t),m_{x_0}(x,t),\hat u_{0}(t))\tilde{u}_{0}(t)\\
		&&-\displaystyle\int\dfrac{\partial f_{0}}{\partial m}(x_{0}(t),m_{x_0}(x,t),\hat u_{0}(t))(x)\tilde m_{x_0}(x,t)dx\bigg]dt+\left\{\dots\right\}dW_0(t).\\
	\end{array}
	\end{equation*}
	Integrating and taking expectation on both sides gives
		\begin{equation*}
		\begin{array}{rl}
			&\mathbb{E}\bigg[(h_{0,x_0}(x_0(T),m_{x_0}(x,T))+\displaystyle 	\int{r}(x,T)h_{1,x_0}(x,x_0(T),m_{x_0}(x,T))dx)\tilde x_0(T)dx\\
			&\displaystyle+\int \Big[\dfrac{\partial h_{0}}{\partial m}(x_{0}(T),m_{x_0}(x,T))(x)+\displaystyle \int {r}(\xi,T)\dfrac{\partial h_1}{\partial m}(\xi,x_0(T),m_{x_0}(x,T))(x)d\xi \Big]\tilde m_{x_0}(x,T)dx\\
			&\displaystyle-\int r(x,T)\Big[ h_{1,x_0}(x,x_0(T),m_{x_0}(x,T))\tilde{x}_0(T)+\displaystyle \int \dfrac{\partial h_1}{\partial m}(x,x_0(T),m_{x_0}(x,T))(\xi) \tilde{m}_{x_0}(\xi,T)d\xi\Big]dx\bigg]\\
			=&\displaystyle\mathbb{E}\int_0^T\bigg\{\displaystyle-\tilde x_0^*(t)f_{0,x_{0}}(x_{0}(t),m_{x_0}(x,t),\hat u_{0}(t))+p^*(t)g_{0,\hat u_{0}}(x_{0}(t),m_{x_0}(x,t),\hat u_{0}(t))\tilde{u}_{0}(t)\\
			&-\displaystyle\int\dfrac{\partial f_{0}}{\partial m}(x_{0}(t),m_{x_0}(x,t),\hat u_{0}(t))(x)\tilde m_{x_0}(x,t)dx\bigg\}dt.\\
		\end{array}
		\end{equation*}
	Finally we consider (\ref{gatJ0}) and obtain
	\begin{equation*}
	\begin{array}{rcl}
			0&=&\mathbb{E} \displaystyle\int_0^T\Big\{f_{0,\hat u_{0}}(x_{0},m_{x_0},\hat u_{0})+p^*g_{0,u_0}(x_{0},m_{x_0},\hat u_{0})\Big\}\tilde{u}_0dt.
	\end{array}
	\end{equation*}
	Since $\tilde{u}_0$ is arbitrary, the control is optimal for the dominating player only if 
		\begin{equation*}
		\begin{split}
			f_{0,u_0}(x_{0},m_{x_0},\hat u_{0})+p^*g_{0,u_0}(x_{0},m_{x_0},\hat u_{0})=0, \text{a.e. t.}
		\end{split}
		\end{equation*}
	Again, we are considering a minimization problem with the first order condition given above. Since the unique existence of a minimizer, $\hat u_0$ is assumed, we conclude that $\hat u_0$ satisfies the infimum
	\begin{equation*}
		f_0(x_0,m_{x_0},\hat u_0)+p \cdot g_0(x_0,m_{x_0},\hat u_0)= \inf_{u_0} \Big\{f_0(x_0,m_{x_0},u_0)+p \cdot g_0(x_0,m_{x_0},u_0)\Big\}. \\
	\end{equation*}
 \qed \end{proof}
	Let $G_0(x,m_{x_0},p)=g_0(x_0,m_{x_0},\hat{u}_0(x_0,m_{x_0},p))$. We then conclude the main result in this section.
\begin{thm}
	 	{\it The necessary condition for Problems \ref{P1}, \ref{P2} and \ref{P3} is provided by the following six equations}
	 	\label{main}
	 	\begin{equation*}
	 	\begin{array}{l}
	 	\left\{
	 	\begin{array}{rll}
			dx_0 &=& G_0(x_0(t),m_{x_0(t)}(t),p(t))dt+\sigma_0(x_0(t))dW_0(t),\\
			x_0(0) &=& \xi_0.\\
			\\
			\dfrac{\partial m_{x_0}}{\partial t}&=&-A_1^{*}m_{x_0}(x,t)-{\normalfont \text{div}}\Big(G_1\Big(x,x_0(t),m_{x_0}(x,t),D\Psi(x,t)\Big)m_{x_0}(x,t)\Big),\\
			m_{x_0}(x,0) &=&\omega(x),\\
			\\
			-\partial_{t}\Psi&=& \Big(H_1\Big(x,x_0(t),m_{x_0}(x,t),D\Psi(x,t)\Big)-A_1\Psi(x,t)\Big)dt  -K_{\Psi}(x,t)dW_0(t), \\
			\Psi(x,T)  &=&h_1(x,x_{0}(T),m_{x_0}(x,T)).\\
		\end{array}
		\right.\\
		\\
		\left\{
		\begin{array}{rll}
			-dp&=&\bigg[H_{0,x_0}(x_0(t),m_{x_0}(x,t),p(t))+\displaystyle \int{r}(x,t)H_{x_{0}}(x,x_{0}(t),m_{x_0}(x,t),D\Psi(x,t))dx\\
			&&+\displaystyle \int G_{1,x_{0}}(x,x_{0}(t),m_{x_0}(x,t),D\Psi(x,t))D{q}^*(x,t)m_{x_0(t)}(x,t)dx\bigg]dt\\
			&&-\sum_{l=1}^{d_0}K_p^l(t)dW_0^l(t) +\sum_{l=1}^{d_0}\sigma_{0,x_{0}}^{l*}(x_0(t))K_p^l(t)dt, \\
			p(T)&=&h_{0,x_0}(x_0(T),m_{x_0}(x,T))+\displaystyle 	\int{r}(x,T)h_{1,x_0}(x,x_0(T),m_{x_0}(x,T))dx; \\
			\\	
			-\partial_{t}{q}&=&\bigg[-A_1{q}(x,t)+\dfrac{\partial H_{0}}{\partial m}(x_{0}(t),m_{x_0}(\xi,t),p(t))(x)\\
			&&+D{q}(x,t)G_1(x,x_{0}(t),m_{x_0}(x,t),D\Psi{(x,t)})\\
			&&+\displaystyle \int D{q}(\xi,t)\dfrac{\partial G_1}{\partial m}(\xi,x_{0}(t),m_{x_0}(\xi,t),D\Psi{(\xi,t)})(x)m_{x_0}(\xi,t)d\xi\\
			&&+\displaystyle \int{r}(\xi,t)\dfrac{\partial H_1}{\partial m}(\xi,x_{0}(t),m_{x_0}(\xi,t),D\Psi{(\xi,t)})(x)d\xi\bigg]dt-K_q(x,t)dW_0(t),\\
			{q}(x,T)&=&\dfrac{\partial h_{0}}{\partial m}(x_{0}(T),m_{x_0}(\xi,T))(x)+\displaystyle \int {r}(\xi,T)\dfrac{\partial h_1}{\partial m}(\xi,x_0(T),m_{x_0}(\xi,T))(x)d\xi;\\
			\\
			\dfrac{\partial{r}}{\partial t}&=&-A_1^{*}{r}(x,t)-\text{\rm div}\bigg[{r}(x,t)H_{1,q}(x,x_{0}(t),m_{x_0}(x,t),D\Psi{(x,t)})\\
			&&+G_{1,q}(x,x_{0}(t),m_{x_0}(x,t),D\Psi{(x,t)})D{q}(x,t)m_{x_0}(x,t)\bigg],\\
			{r}(x,0)&=&0.
	 	\end{array}
	 	\right.
	 	\end{array}
	 	\end{equation*}
\end{thm}
\section{Linear Quadratic Case}\label{lqcase}
	In this section we present a special case of the problem in the Linear Quadratic setting in which both necessary and sufficient condition could be established. Suppose that the state evolutions of the processes $ x_0(t), x_1(t) $ are described by 
	\begin{equation*}
	\left\{
	\begin{array}{rcl}
		dx_0 &=& \Big(A_0x_0(t)+B_0z(t)+C_0u_0(x_0(t),t)\Big)dt+\sigma_0dW_0(t),\\
		x_0(0) &=& \xi_0.\\
		dx_1 &=& \Big(A_1x_1(t)+B_1z(t)+C_1u_1(x_1(t),t)+Dx_0(t)\Big)dt+\sigma_1 dW_1(t),\\
		x_1(0) &=& \xi_1.\\
	\end{array}
	\right.
	\end{equation*}
	To simplify,  the matrices $A_0,C_0,B_0,\sigma_0$; $A_1,C_1,B_1,D,\sigma_1$ are assumed to be constant though the case with time dependent and deterministic function are similar. 
	For if the dominating player did not exist, it is customary to consider $z(t)$ as deterministic, and the equilibrium condition is 
		\begin{equation}
		\label{fxpt_1}
			z(t) = \mathbb{E}x_1(t).
		\end{equation}
	We can first find an optimal stochastic control for the representative agent given $z$, then solve for the fixed point equation (\ref{fxpt_1}).
	However, one cannot assume $z(t)$ to be deterministic when the dominating player exists, which induces a two-layer problem. Since the dominating player can directly influence the mean field term in the present setting, $z(t)$ should be adapted to the filtration $ \mathcal{F}_t^0$. The equilibrium condition hence is 
		\begin{equation}
		\label{eql}
			z(t) = \mathbb{E}^{\mathcal{F}_t^0}x_1(t)=\int \xi m_{x_0(t)}(\xi,t)d\xi.
		\end{equation} 	
	We define control problems for both the dominating player and the representative agent and fist solve for the control problem of the representative agent as if both $x_0(t)$ and $z(t)$ as exogenous. The next problem is to solve the equilibrium condition (\ref{eql}) as a fixed point property. Finally we solve for the control problem of the dominating player, but now $z(t)$ is regarded as endogenous.

	For any vector $v$ and matrix $M$ with appropriate dimensions, we write the inner product $v^*Mv$ as $|v|^2_M$ for simplicity. Problems \ref{P1}, \ref{P2} and \ref{P3} can now be rewritten in the present Linear Quadratic framework as follows:
	\begin{pblm}{\bf Control of the Representative Agent}\label{P21}\\
	Given the process $x_0$ and $\kappa$, find a control $u_1\in\mathcal{A}_1$ which minimizes the cost functional
	\begin{equation*}
	\begin{split}
		J_1(u_1, x_0, \kappa) &= \mathbb{E}\bigg[\int_{0}^{T}(| x_1(t)-E_1\kappa(t)-Fx_0(t)-\zeta_1|_{Q_1}^{2}+u^{*}_1(t)R_1 u_1(t))dt\\
		&\qquad\qquad+| x_1(T)-\bar{E}_1\kappa(T)-\bar{F}x_0(T)-\bar{\zeta}_1|_{\bar{Q}_1}^{2}\bigg].
	\end{split}	 
	\end{equation*}
	\end{pblm}
	\begin{pblm}{\bf Equilibrium Condition}\label{P22}\\
	Let $x_1:=x_{1,\kappa}$ be the trajectory of the representative agent with the optimal control found in Problem \ref{P21}.	Find the process $z(t)$ such that the fixed point property is satisfied
	\begin{equation*}
		z(t) = \mathbb{E}^{\mathcal{F}_t^0}x_{1,z}(t),
	\end{equation*}
	\end{pblm}
	\begin{pblm}{\bf Control of the Dominating Player}\label{P23}\\
	Find a control $u_0\in\mathcal{A}_0$ which minimizes the cost functional
	\begin{equation*}
	\begin{split}
		J_0(u_0) &= \mathbb{E}\bigg[\int_{0}^{T}(| x_0(t)-E_0z(t)-\zeta_0|_{Q_0}^{2}+u_0^{*}(t)R_0u_0(t))dt\\
		&\qquad\qquad+|x_0(T)-\bar{E}_0z(T)-\bar{\zeta}_0|_{\bar{Q}_0}^{2}\bigg],
	\end{split}
	\end{equation*}
	where $z$ is the solution given in Problem \ref{P22}.
	\end{pblm} 	
	For simplicity, $E_0, F, R_0, R_1, E_1, \zeta_0, \zeta_1, Q_0, Q_1$ are constant matrices and vectors; $Q_0, Q_1, R_0, R_1$ are positive symmetric and
	invertible. Note that Problems \ref{P21} and \ref{P23} are strictly convex quadratic and coercive, we can write the stochastic  principle.
	\begin{lem}{\bf (Control of Representative Agent)}\label{lem21} \\
		Problem \ref{P21} is uniquely solvable and the optimal control $\hat u_1(t)$ is $-R_1^{-1}C_1^*n(t)$, where $n$ satisfies the adjoint process
		\begin{equation*}
		\begin{split}
		\left \{ \begin{array}{rlc}
		-dn &= \Big(A_1^*n(t)+Q_1(x_1(t)-E_1\kappa(t)-Fx_0(t)-\zeta_1)\Big)dt-Z_{n,0}(t)dW_0(t)-Z_{n,1}(t)dW_1(t),\\
		n(T) &= \bar{Q}_1\Big(x_1(T)-\bar{E}_1\kappa(T)-\bar{F}x_0(T)-\bar{\zeta}_1\Big).
		\end{array}\right.
		\end{split}
		\end{equation*}
	\end{lem}
	\begin{proof}
		Consider a perturbation of the optimal control $\hat u_1+\theta \tilde{u}_1$, where $\tilde{u}_1$ is adapted to the filtration $\mathcal{G}_t$. The original state $x_1$ becomes $x_1+\theta \tilde{x}_1$ with
		\begin{equation*}
		\begin{split}
		\left \{ \begin{array}{rlc}
		d\tilde{x}_1 &= \Big(A_1\tilde{x}_1(t)+C_1\tilde{u}_1(t)\Big)dt,\\
		\tilde{x}_1(0) &= 0.
		\end{array}\right.
		\end{split}
		\end{equation*}
		the optimality of $\hat u_1$ is expressed by the following Euler condition
		\begin{equation*}
		\begin{split}
			0&=\at{\dfrac{d}{d\theta}}{\theta=0}J_1(\hat u_1+\theta\tilde{u}_1,x_0,\kappa) \\
			&=\mathbb{E}\bigg[\int_0^T [\tilde{x}_1^*(t)Q_1(x_1(t)-E_1\kappa(t)-Fx_0(t)-\zeta_1)+\tilde{u}_1^*(t)R_1\hat u_1(t)]dt\\
			&\qquad\qquad+\tilde{x}_1^*(T)\bar{Q}_1(x_1(T)-\bar{E}_1\kappa(T)-\bar{F}x_0(T)-\bar{\zeta}_1)\bigg]\\
		\end{split}
		\end{equation*} 
		On the other hand, we have
		\begin{equation*}
		\begin{split}
			d(n^*\tilde{x}_1)&=\Big(n^*(t)C_1\tilde{u}_1(t)-\tilde{x}_1^*(t)Q_1(x_1(t)-E_1\kappa(t)-Fx_0(t)-\zeta_1)\Big)dt\\
			&\qquad+\tilde{x}_1(t)Z_{n,0}(t)dW_0(t)+\tilde{x}_1(t)Z_{n,1}(t)dW_1(t).
		\end{split}
		\end{equation*}
		Integrate both side and take expectation, combining with the Euler condition, it becomes
		\begin{equation*}
			\mathbb{E}\int_0^T(n^*(t)C_1+\hat u_1^*(t)R_1)\tilde{u}_1(t)dt = 0,\quad \forall \tilde{u}_1(\cdot)\text{ adpated to } \mathcal{G},
		\end{equation*} 
		which implies, together with an application of tower property,
		\begin{equation}
			\hat u_1(t) = -R_1^{-1}C_1^*n(t).
		\end{equation} 
		Therefore the stochastic maximum principle for the representative agent is expressed by the system
		\begin{equation}
		\label{x1n}
		\left\{
		\begin{array}{rlc}
			dx_1 &= \Big(A_1x_1(t)+B_1\kappa(t)-C_1R_1^{-1}C_1^*n(t)+Dx_0(t)\Big)dt+\sigma_1 dW_1(t), \\
			x_1(0) &= \xi_1. \\
			-dn &= \Big(A_1^*n(t)+Q_1(x_1(t)-E_1\kappa(t)-Fx_0(t)-\zeta_1)\Big)dt-Z_{n,0}(t)dW_0(t)-Z_{n,1}(t)dW_1(t),\\
			n(T) &= \bar{Q}_1\Big(x_1(T)-\bar{E}_1\kappa(T)-\bar{F}x_0(T)-\bar{\zeta}_1\Big).
		\end{array}
		\right.
		\end{equation}
		and hence for every exogenous pair $(x_0, \kappa)$, this system defines a unique pair $(x_1,n)$. By the convexity and coerciveness of the cost functional, $\hat u_1$ is uniquely defined and the sufficient condition is automatically satisfied. 
	  \qed \end{proof}
	\begin{rk}
	The optimal control has the representation $-R_1^{-1}C_1^*n(t)=-R_1^{-1}C_1^*(P_tx_1(t)+g(t))$, where $P$ satisfies the symmetric Riccati equation
	\begin{equation*}
		\left\{
		\begin{array}{rl}
			&\dot{P_t}+P_tA_1+A_1^*P_t-P_tC_1R_1^{-1}C_1^*P_t+Q_1=0,\\
			&P_T=\bar{Q}_1;\\
		\end{array}
		\right.\\
	\end{equation*}	
	and $g$ satisfies the BSDE
	\begin{equation*}
		\left \{ 
		\begin{array}{rcl}
			-dg &=& \Big((A_1-C_1R_1^{-1}C_1^*P_t)^*g(t)+(P_tB_1-Q_1E_1)+(P_tD-Q_1F)x_0(t) \kappa(t)-Q_1\zeta_1\Big)dt\\
			&&\qquad-Z_{g,0}(t)dW_0(t)-Z_{g,1}(t)dW_1(t),\\
			g(T)&=&\bar{Q}_1\Big(-\bar{E}_1\kappa(T)-\bar{F}x_0(T)-\bar{\zeta}_1\Big).\\
		\end{array}
		\right.
	\end{equation*}		
	\end{rk}
	To obtained the equilibrium condition stated in Problem (\ref{P22}), we simply take expectation conditional on $\mathcal{F}_t^0$ on both sides of Equation (\ref{x1n}). By requiring $\mathbb{E}^{\mathcal{F}^0_t}x_1(t)=z(t)$, and replace $\kappa(t)$ by $z(t)$, we have the following system, which is analogical to the SHJB-FP (\ref{HJBFP}) in Section \ref{gen}.
	\begin{equation}
		\left\{
		\begin{array}{rlc}
			dz &= \Big((A_1+B_1)z(t)-C_1R_1^{-1}C_1^*m(t)+Dx_0(t)\Big)dt+\sigma_1 dW_1(t), \\
			z_1(0) &= \mathbb{E}[\xi_1]. \\
			-dm &= \Big(A_1^*m(t)+Q_1(I-E_1)z(t)-Q_1Fx_0(t)-Q_1\zeta_1\Big)dt-Z_m(t)dW_0(t),\\
			m(T) &= \bar{Q}_1(I-\bar{E}_1)z(T)-\bar{Q}_1\bar{F}x_0(T)-\bar{Q}_1\bar{\zeta}_1.
		\end{array}
		\right.
	\end{equation}	
	We next proceed on the problem of the dominating player.
	\begin{lem}{\bf (Control of the Dominating Player)}\label{lem23}\\ 
	Problem \ref{P23} is uniquely solvable and the optimal control $\hat u_0(t)$ is $-R_0^{-1}C_0^*p(t)$, where $p$ satisfies
	\begin{equation}
	\left\{ \begin{array}{rcl}
		-dp &=& \Big(A_0^*{p}(t)+D^*q(t)-(Q_1F)^*r(t)+Q_0(x_0(t)-E_0z(t)-\zeta_0)\Big)dt-Z_p(t)dW_0(t),\\
		{p}(T) &=&\bar{F}^* \bar{Q}_1 r(T)+ \bar{Q}_0(x_0(T)-\bar{E}_0z(T)-\bar{\zeta}_0).\\
		-dq &=& \Big((A_1+B_1)^*q(t)+B_0^*{p}(t)+(I-E_1)^* Q_1^* r(t)-E_0^*Q_0(x_0(t)-E_0z(t)-\zeta_0)\Big)dt-Z_q(t)dW_0(t),\\
		q(T) &=& (I-\bar{E}_1)^*\bar{Q}_1 r(T)-\bar{E}_0^*\bar{Q}_0(x_0(t)-\bar{E}_0z(t)-\bar{\zeta}_0).\\
		dr &=& \Big(A_1r(t)-C_1R_1^{-1}C_1^*q(t)\Big)dt,\\
		{r}(0) &=& 0.
	\end{array}\right.
	\end{equation}
	\end{lem}	
	\begin{proof}
		Consider $\hat u_0+\theta\tilde{u}_0$ the perturbation of the optimal control, where $\tilde{u}_0$ is adapted to the filtration $\mathcal{F}_t^0$. The original states $x_0$, $z$, $m$ become $x_0+\theta\tilde{x}_0$, $z+\theta\tilde{z}$, $m+\theta\tilde{m} $ with
		\begin{equation*}
		\left\{ \begin{array}{rcl} 
			d\tilde{x}_0 &=& \Big(A_0\tilde{x}_0(t)+B_0\tilde{z}(t)+C_0\tilde{u}_0(t)\Big)dt,\\
			\tilde{x}(0) &=& 0.\\
			d\tilde{z}&=& \Big((A_1+B_1)\tilde{z}(t)-C_1R_1^{-1}C_1^* \tilde{m}(t)+D\tilde{x}_0(t)\Big)dt,\\
			\tilde{z}(0) &=& 0.\\
			-d\tilde{m}&=&\Big(A_1^*\tilde{m}(t)+Q_1(I-E_1)\tilde{z}(t)-Q_1F\tilde{x}_0(t)\Big)dt-Z_{\tilde{m}}dW_0(t)\\
			\tilde m(T)&=&\bar{Q}_1(I-\bar{E}_1)\tilde{z}(T)-\bar{Q}_1\bar{F}\tilde{x}_0(T).\\
		\end{array}\right.
		\end{equation*}
		The corresponding maximum principle for $\hat u_0$ is
		\begin{equation}
		\label{opu_0}
		\begin{split}
			0&=\at{\dfrac{d}{d\theta}}{\theta=0}J_0(\hat u_0+\theta\tilde{u}_0) \\
			&=\mathbb{E}\bigg[\int_{0}^T[(\tilde{x}_0(t)-E_0\tilde{z}(t))^* Q_0 (x_0(t)-E_0z(t)-\zeta_0)+\hat u_0^*R_0\tilde{u}_0]dt\\
			&\qquad\qquad+(\tilde{x}_0(T)-\bar{E}_0\tilde{z}(T))^* Q_0 (x_0(T)-\bar{E}_0z(T)-\bar{\zeta_0})\bigg]
		\end{split}
		\end{equation} 
		On the other hand, we can easily check that		\begin{equation*}
		\begin{split}
			d({p}^*\tilde{x}_0+q^*\tilde{z}-r^* \tilde{m})=\Big({p}^*(t)C_0\tilde{u}_0(t)-(\tilde{x}_0(t)-E_0\tilde{z}(t))^*Q_0(x_0(t)-E_0z(t)-\zeta_0)\Big)dt+\{\dots\}dW_0(t).
		\end{split}
		\end{equation*}
		Integrating and also taking expectation on both sides of the last equation, together with an application of (\ref{opu_0}), we deduce that 
		\begin{equation*}
		\begin{split}
			\mathbb{E}\int_0^T ({p}^*C_0+\hat u_0^*R_0)\tilde{u}_0 dt = 0,\quad \forall \tilde{u}_0\text{ adapted to }\mathcal{F}_t^0,
		\end{split}
		\end{equation*} 
		which implies the desired result by again the application of tower property,
		\begin{equation}
		\label{u0}
			\hat u_0(t) = -R_0^{-1}C_0^*p(t)(t).
		\end{equation}
	 \qed \end{proof}
	
	Summarizing the results we obtained so far, we present the main theorem in this section.
	\begin{thm}
	\label{main2}
	The necessary and sufficient conditions for the unique existence of the solution to Problems \ref{P21}, \ref{P22} and \ref{P23} are described by the following six equations in matrix form
	\begin{equation}
	\label{FBSDE}
	\begin{array}{rl}	
		&\left\{
		\begin{array}{rcl}
		d\begin{pmatrix} x_0 \\r \\z \end{pmatrix} &=&	\left\{\begin{pmatrix} A_0 & 0 & B_0 \\ 0 & A_1 & 0 \\ D & 0 & A_1+B_1 \end{pmatrix}\begin{pmatrix} x_0(t) \\r(t) \\z(t) \end{pmatrix}\right.\\
		\quad&&\left.-\begin{pmatrix} C_0R_0^{-1}C_0^* & 0 & 0 \\ 0 & C_1R_1^{-1}C_1^* & 0\\ 0 & 0 & C_1R_1^{-1}C_1^* \end{pmatrix}\begin{pmatrix} p(t) \\ q(t) \\ m(t) \end{pmatrix}\right\}dt+\begin{pmatrix} \sigma_0 \\ 0 \\ 0   \end{pmatrix}dW_0,\\
		\begin{pmatrix} x_0(0) \\r(0) \\z(0) \end{pmatrix}&=&\begin{pmatrix} \xi_0 \\0 \\ \mathbb{E}[\xi_1] \end{pmatrix}.\\
		\end{array}
		\right. \\
		\\
		&\left\{
		\begin{array}{rcl}
		-d\begin{pmatrix} p \\q \\m \end{pmatrix} &=&	\left\{\begin{pmatrix} A_0^* & D^* & 0 \\ B_0^* & A_1^*+B_1^* & 0 \\ 0 & 0 & A_1^* \end{pmatrix}\begin{pmatrix} p(t) \\q(t) \\m(t) \end{pmatrix}\right.\\
		&&\left. +\begin{pmatrix} Q_0 & -(Q_1F)^* & -Q_0E_0 \\ -(Q_0E_0)^* & (Q_1(I-E_1))^* & E_0^* Q_0 E_0\\ -Q_1 F & 0 & Q_1(I-E_1) \end{pmatrix}\begin{pmatrix} x_0(t) \\ r(t) \\ z(t) \end{pmatrix}+\begin{pmatrix} -Q_0\zeta_0 \\ E_0^*Q_0\zeta_0 \\ -Q_1\zeta_1 \end{pmatrix}\right\}dt-
		\begin{pmatrix}
		Z_p(t)\\Z_q(t)\\Z_m(t)
		\end{pmatrix}dW_0(t),\\
		\begin{pmatrix} p(T) \\q(T) \\m(T) \end{pmatrix}&=&\begin{pmatrix} \bar{Q}_0 & -(\bar{Q}_1\bar{F})^* & -\bar{Q}_0\bar{E}_0 \\ -\bar{Q}_0\bar{E}_0)^* & (\bar{Q}_1(I-\bar{E}_1))^* & \bar{E}_0^* \bar{Q}_0 \bar{E}_0\\ \bar{Q}_1 F & 0 & \bar{Q}_1(I-\bar{E}_1) \end{pmatrix}\begin{pmatrix} x_0(t) \\ r(t) \\ z(t) \end{pmatrix}+\begin{pmatrix} -\bar{Q}_0\bar{\zeta}_0 \\ \bar{E}_0^*\bar{Q}_0\bar{\zeta}_0 \\ -\bar{Q}_1\bar{\zeta}_1 \end{pmatrix}.\\
		\end{array}
		\right.
	\end{array}
	\end{equation}
	\end{thm}
	\begin{rk}
	One can easily compare these six equations with those stated in Theorem \ref{main}. We obtain the same results by applying the general theory, however, it is more convenient to acquire these six equations directly under the Linear Quadratic setting, which also illuminates the power of using the principle of separation.
	On comparison with the intermediary result obtained in \cite{HN2}. The latter work did not  take account of the third adjoint equation $r$ since it fails to consider the impact on $m$ with respect to the change of the control of the dominating player.
	\end{rk}
\section{Fixed Point Problem}\label{fxpt}
	In this section, we provide a sufficient condition, which solely depends on the coefficients of the mean field game system, for the unique existence of the solution to Problems \ref{P21}, \ref{P22} and \ref{P23} by means of tackling a {\it non-symmetric Riccati equation}. To facilitate our argument, we define
	\begin{equation*}
		{\bf x}:=\begin{pmatrix} x_0 \\ r \\ z \end{pmatrix}; \quad
		{\bf p}:=\begin{pmatrix} p \\ q \\ m \end{pmatrix}; \quad
		{\bf Z}:=\begin{pmatrix} Z_p \\ Z_q \\ Z_m \end{pmatrix}.
	\end{equation*}
	Hence we can write (\ref{FBSDE}) as 
	\begin{equation}
	\label{FBSDE2}
	\begin{array}{rl}
		&\left\{
		\begin{array}{rcl}
			d{\bf x}&=&\Big((\mathcal{A}+\mathcal{B}){\bf x}(t)-\mathcal{C}{\bf p}(t)\Big)dt+\sigma dW_0(t),\\
			{\bf x}(0)&=&\xi.\\
		\end{array}
		\right.
		\\
		&\left\{
		\begin{array}{rcl}
			-d{\bf p}&=&\Big(\mathcal{A}^* {\bf p}(t) + (\mathcal{Q}+\mathcal{S}){\bf x}(t)+{\bf k}\Big)dt-{\bf Z}(t)dW_0(t),\\
			{\bf p}(T)&=&(\bar{\mathcal{Q}}+\bar{\mathcal{S}}){\bf x}(T)+\bar{\bf k},\\
		\end{array}
		\right.
	\end{array}
	\end{equation}
	where
	\begin{equation*}
	\begin{array}{l}
		\mathcal{A}:=\begin{pmatrix} A_0 & B_0 & 0 \\ D & A_1+B_1 & 0 \\ 0 & 0 & A_1 \end{pmatrix}; \quad
		\mathcal{B}:=\begin{pmatrix} 0 & -B_0 & B_0 \\ -D & -B_1 & 0 \\ D & 0 & B_1 \end{pmatrix}; \quad
		\mathcal{C}:=\begin{pmatrix} C_0R_0^{-1}C_0^* & 0 & 0 \\ 0 & C_1R_1^{-1}C_1^* & 0\\ 0 & 0 & C_1R_1^{-1}C_1^* \end{pmatrix};\\
		
		{\sigma}:=\begin{pmatrix} \sigma_0 \\ 0 \\ 0 \end{pmatrix}; \quad 
		\xi:=\begin{pmatrix} \xi _0 \\ 0 \\ 0 \end{pmatrix}; \quad 
		\mathcal{Q}+\mathcal{S}:=\begin{pmatrix} Q_0 & -F^*Q_1 & -Q_0E_0 \\ -E_0^*Q_0 & \quad(I-E_1)^*Q_1^* & E_0^* Q_0 E_0\\ -Q_1 F & 0 & \quad Q_1(I-E_1) \end{pmatrix}; \quad
		{\bf k}:=\begin{pmatrix} -Q_0\zeta_0 \\ E_0^*Q_0\zeta_0 \\ -Q_1\zeta_1 \end{pmatrix};\\
		\bar{\mathcal{Q}}+\bar{\mathcal{S}}:=\begin{pmatrix} \bar{Q}_0 & -(\bar{Q}_1\bar{F})^* & -\bar{Q}_0\bar{E}_0 \\ -\bar{Q}_0\bar{E}_0)^* & (\bar{Q}_1(I-\bar{E}_1))^* & \bar{E}_0^* \bar{Q}_0 \bar{E}_0\\ \bar{Q}_1 F & 0 & \bar{Q}_1(I-\bar{E}_1) \end{pmatrix}; \quad 
		\bar{\bf k}:=\begin{pmatrix} -\bar{Q}_0\bar{\zeta}_0 \\ \bar{E}_0^*\bar{Q}_0\bar{\zeta}_0 \\ -\bar{Q}_1\bar{\zeta}_1 \end{pmatrix},
	\end{array}
	\end{equation*}
	where $\mathcal{Q}$ and $\bar{\mathcal{Q}}$ are positive matrices. Consider the following {\it non-symmetric Riccati equation}
	\begin{equation}
	\label{NSRE}
	\left\{
	\begin{array}{l}
		\dot{\Gamma}+\mathcal{A}^*\Gamma_t+\Gamma_t(\mathcal{A}+\mathcal{B})-\Gamma_t \mathcal{C} \Gamma_t+(\mathcal{Q}+\mathcal{S})=0,\\
		\Gamma_T=\bar{\mathcal{Q}};\\
	\end{array}
	\right.
	\end{equation}
	and the backward ODE
	\begin{equation}
	\left\{
	\begin{array}{rcl}
		-d{\bf g}&=&\Big((\mathcal{A}^*-\Gamma_t\mathcal{B}){\bf g}(t)+{\bf k}\Big)dt,\\
		{\bf g}(T)&=&\bar{\bf k}.
	\end{array}
	\right.
	\end{equation}	
	It is easy to check that ${\bf p}(t)=\Gamma_t {\bf x}(t)+{\bf g}(t)$.
	With respect to this affine form, the forward backward equation (\ref{FBSDE2}) admits a unique solution if and only if (\ref{NSRE}) admits a unique solution. In accordance with Theorem 2.4.3 in Ma and Young \cite{FBSDE} or \cite{AB}, we have the following proposition.
	\begin{prop}
		Suppose the following forward-backward ordinary differential equations
		\begin{equation*}
			\begin{array}{rl}
				&\left\{
				\begin{array}{rcl}
					\dfrac{dX}{dt}&=&(\mathcal{A}+\mathcal{B})X(t)-\mathcal{C}Y(t)\\
					X(0)&=&0.\\
				\end{array}
				\right.
				\\
				&\left\{
				\begin{array}{rcl}
					-\dfrac{dY}{dt}&=&\mathcal{A}^*Y(t) + (\mathcal{Q}+\mathcal{S})X(t),\\
					Y(T)&=&(\bar{\mathcal{Q}}+\bar{\mathcal{S}})X(T).\\
				\end{array}
				\right.
			\end{array}
		\end{equation*}
		admits a unique solution for any $t_0 \in [0,T]$. Then there is a unique solution of (\ref{NSRE}).
	\end{prop}
	Our next theorem concludes the results in this section.
	\begin{thm}
		Let $\phi(s,t)$ be the fundamental solution to $\mathcal{A}$. Suppose that
		\begin{equation*}
			\Big(1+\sqrt{T}\|\phi\|_T\cdot\|\mathcal{B}\mathcal{Q}^{-\frac{1}{2}}\|\Big)\Big(1+ N(S)\Big)<2,
		\end{equation*}
		where $\|\cdot\|$ stands for usual Euclidean norm.	Then there exists a unique solution of equation (\ref{FBSDE}), and hence a unique (mean field) equilibrium exists. Here,
		\begin{equation*}
			\|\phi\|_T:=\sup_{0\leq t \leq T}\sqrt{\|\phi^*(T,t)\bar{\mathcal{Q}}^{\frac{1}{2}}\|^2+\int_t^T\|\phi^*(s,t)\mathcal{Q}^{\frac{1}{2}}\|^2 ds}
		\end{equation*}
		and 
		\begin{equation*}
			N(S)=\max\{\|\bar{\mathcal{Q}}^{-\frac{1}{2}}\bar{\mathcal{S}}\bar{\mathcal{Q}}^{-\frac{1}{2}}\|,
			\|\mathcal{Q}^{-\frac{1}{2}}\mathcal{S}\mathcal{Q}^{-\frac{1}{2}}\|\}
		\end{equation*}
	\end{thm}
	\begin{proof}
		Let $x,y$ be elements in the Hilbert Space $\mathcal{H}^2([0,T];\mathbb{R}^{n_0+n_1+n_1})$ endowed with the inner product
		\begin{equation*}
			\langle x,y \rangle_{\mathcal{H}}=x^*(T) \bar{\mathcal{Q}} y(T)+\int_0^T x^*(s) \mathcal{Q} y(s) ds
		\end{equation*}
		Furthermore, $\|\cdot\|_{\mathcal{H}}:=|\langle \cdot,\cdot \rangle|_{\mathcal{H}}^{\frac{1}{2}}$ stands for the induced norm under this inner product. We consider the forward backward ordinary differential equation
		\begin{equation}
		\label{FBODE}
			\begin{array}{rl}
				&\left\{
				\begin{array}{rcl}
					\dfrac{dX}{dt}&=&\mathcal{A}X(t)-\mathcal{C}Y(t)+\mathcal{B}x(t)\\
					X(0)&=&0.\\
				\end{array}
				\right.
				\\
				&\left\{
				\begin{array}{rcl}
					-\dfrac{dY}{dt}&=&\mathcal{A}^*Y(t) + \mathcal{Q}X(t) + \mathcal{S}x(t),\\
					Y(T)&=&\bar{\mathcal{Q}}X(T)+\bar{\mathcal{S}}x(T).\\
				\end{array}
				\right.
			\end{array}
		\end{equation}
		Observe that both $\mathcal{C}$ and $\mathcal{Q}$ are positive definite, Equation (\ref{FBODE}) corresponds a well-defined (deterministic) control problem. Hence, $x \mapsto X$ is well defined in $\mathcal{H}^2$. It suffices to show that this mapping is indeed a contraction. Consider the inner product 
		\begin{equation*}
			\dfrac{d}{dt}(X^*Y)=-Y^*(t)\mathcal{C}Y(t)+Y^*(t)\mathcal{B}x(t)-X^*(t)\mathcal{Q}X(t)-X^*(t)\mathcal{S}x(t).
		\end{equation*}
		Taking integration on $[0,T]$ yields
		\begin{equation*}
			X^*(T)\bar{\mathcal{Q}}X(T)+\int_0^T X^*(t)\mathcal{Q}X(t) dt=X(T)^*\bar{\mathcal{S}}x(T)+\int_0^T -Y^*(t)\mathcal{C}Y(t)+Y^*(t)\mathcal{B}x(t)-X^*(t)\mathcal{S}x(t) dt.
		\end{equation*}
		By the positivity of $\mathcal{C}$, and Cauchy-Schwarz inequality, we have
		\begin{equation}
		\label{FXPT1}
		\begin{array}{rcl}
			\|X\|_{\mathcal{H}}^2&\leq& \|\bar{\mathcal{Q}}^{-\frac{1}{2}}\bar{\mathcal{S}}\bar{\mathcal{Q}}^{-\frac{1}{2}}\|\cdot \Big(X^*(T)\bar{\mathcal{Q}}^{\frac{1}{2}}\cdot \bar{\mathcal{Q}}^{\frac{1}{2}} x(T)\Big)\\
			&&\quad+\|\mathcal{Q}^{-\frac{1}{2}}\mathcal{S}\mathcal{Q}^{-\frac{1}{2}}\|\cdot \displaystyle\int_0^T X^*(t)\mathcal{Q}^{\frac{1}{2}}\cdot \mathcal{Q}^{\frac{1}{2}}x(t) dt+\displaystyle\int_0^TY^*(t)\mathcal{B}x(t)dt\\
			&\leq&N(S)\Big(\|X\|_{\mathcal{H}}\cdot \|x\|_{\mathcal{H}}\Big)+\displaystyle\int_0^TY^*(t)\mathcal{B}x(t)dt.\\
		\end{array}
		\end{equation}
		On the other hand, for $t\in[0,T]$, we have 
		\begin{equation*}
		\begin{split}
			Y(t)&=\phi^*(T,t)\Big(\bar{\mathcal{Q}}X(T)+\bar{\mathcal{S}}x(T)\Big)+\int_t^T \phi^*(s,t)\Big(\mathcal{Q}X(s) + \mathcal{S}x(s)\Big)ds\\
			&=\phi^*(T,t)\bar{\mathcal{Q}}^{\frac{1}{2}}\Big(\bar{\mathcal{Q}}^{\frac{1}{2}}X(T)+\bar{\mathcal{Q}}^{-\frac{1}{2}}\bar{\mathcal{S}}x(T)\Big)+\int_t^T \phi^*(s,t)\mathcal{Q}^{\frac{1}{2}}\Big(\mathcal{Q}^{\frac{1}{2}}X(s) + \mathcal{Q}^{-\frac{1}{2}}\mathcal{S}x(s)\Big)ds\\
		\end{split}
		\end{equation*}		
		which implies
		\begin{equation}
		\label{FXPT2}
			\sup_{t\in[0,T]}\|Y_t\|\leq\|\phi\|_T\Big(\|X\|_{\mathcal{H}}+N(S)\|x\|_{\mathcal{H}}\Big).		
		\end{equation}
		Combining Equations (\ref{FXPT1}) and (\ref{FXPT2}) yields
		\begin{equation*}
		\begin{array}{rcl}
			\|X\|_{\mathcal{H}}^2&\leq&N(S)\Big(\|X\|_{\mathcal{H}}\cdot \|x\|_{\mathcal{H}}\Big)+\sqrt{T}\|\phi\|_T\Big(\|X\|_{\mathcal{H}}+N(S)\|x\|_{\mathcal{H}}\Big)\|\mathcal{B}Q^{\frac{1}{2}}\|\|x\|_{\mathcal{H}}
		\end{array}
		\end{equation*}
		which shows that $x \mapsto X$ is a contraction if 
		\begin{equation*}
		\Big(1+\sqrt{T}\|\phi\|_T\cdot\|\mathcal{B}\mathcal{Q}^{-\frac{1}{2}}\|\Big)\Big(1+ N(S)\Big)<2.
		\end{equation*}
	 \qed \end{proof}
\section{Conclusion}
In this paper, by adopting adjoint equation approach, we provide the general theory and discuss the necessary condition for optimal controls for both the dominating player and the representative agent, and study the corresponding fixed point problem in relation to the equilibrium condition. A convenient necessary and sufficient condition has been provided under the Linear Quadratic setting; in particular, a illuminative sufficient condition, which only involves the coefficient of the mean field game system, for the unique existence of the equilibrium control has been given. Finally, proof of the convergence result of finite player game to mean field counterpart is provided in Appendix. Applications of the present model in connection with central bank lending and systematic risk in financial context will be provided in the future work.

\section{Acknowledgements}
The first author-Alain Bensoussan acknowledges the financial support of the Hong Kong RGC GRF 500113 and the National Science Foundation under grant DMS 1303775. The second author-Michael Chau acknowledges the financial support from the Chinese University of Hong Kong, and the present work constitutes a part of his work for his postgraduate dissertation. The third author-Phillip Yam acknowledges the financial support from The Hong Kong RGC GRF 404012 with the project title: Advanced Topics In Multivariate Risk Management In Finance And Insurance, The Chinese University of Hong Kong Direct Grants 2010/2011 Project ID: 2060422 and 2011/2012 Project ID: 2060444. Phillip Yam also expresses his sincere gratitude to the hospitality of both Hausdorff Center for Mathematics and Hausdorff Research Institute for Mathematics of the University of Bonn during the preparation of the present work.

\appendix
\section{Appendix}
\subsection{$\epsilon$-Nash Equilibrium}
We now establish that the solutions of Problems \ref{P1} and \ref{P2} is an $\epsilon$-Nash Equilibrium. Suppose that there are $N$ representative agents behaving in similar manner, so that the state of the dominating player and the $i$-th agent satisfies the following SDE respectively:
\begin{equation}
\label{EMPIRICAL}
\begin{array}{l}
\left\{
\begin{array}{rcl}
	dy_0&=&g_0\Big(y_0(t),\dfrac{1}{N}\displaystyle\sum_{j=1}^N \delta_{y_1^j(t)},u_0(t)\Big)dt+\sigma_0\Big(y_0(t)\Big)dW_0(t),\\
	y_0(0)&=&\xi_0.\\
\end{array}
\right.\\
\\
\left\{
\begin{array}{rcl}
	dy_1^i&=&g_1\Big(y_1^i(t),y_0(t),\dfrac{1}{N-1}\displaystyle\sum_{j=1,j\neq i}^N \delta_{y_1^j(t)},u_1^i(t)\Big)dt+\sigma_1\Big(y_1^i(t)\Big)dW_1^i(t),\\
	y_1^i(0)&=&\xi_1^i.\\
\end{array}
\right.
\end{array}
\end{equation}
where $\delta_y$ is Dirac measure with a unit mass at $y$. We call Equation (\ref{EMPIRICAL}) the {\it empirical system}. The corresponding objective functional for the $i$-th agent is:
\begin{equation*}
\begin{array}{rcl}
	\mathcal{J}^{N,i}({\bf u})=\mathbb{E}\bigg[\displaystyle\int_0^Tf_1\Big(y_1^i(t),y_0(t),\dfrac{1}{N-1}\displaystyle\sum_{j=1, j \neq i}^N\delta_{y_1^j(t)},u^i(t)\Big)dt+h_1\Big(y_1^i(T),y_0(T),\dfrac{1}{N-1}\displaystyle\sum_{j=1, j \neq i}^N\delta_{y_1^j(T)}\Big)\bigg],
\end{array}
\end{equation*}
where ${\bf u}=(u_1^1,u_1^2,\dots,u_1^N)$. We expect that when $N \rightarrow \infty$, the hypothetical approximation is described by (\ref{SDE}), that is:
\begin{equation}
\label{MEANFIELD}
\begin{array}{l}
\left\{
\begin{array}{rcl}
	dx_0&=&g_0\Big(x_0(t),m_{x_0(t)},u_0(t)\Big)dt+\sigma_0\Big(x_0(t)\Big)dW_0(t),\\
	x_0(0)&=&\xi_0.\\
\end{array}
\right.\\
\\
\left\{
\begin{array}{rcl}
	dx_1^i&=&g_1\Big(x_1^i(t),x_0(t),m_{x_0(t)},u_1^i(t)\Big)dt+\sigma_1\Big(x_1^i(t)\Big)dW_1^i(t),\\
	x_1^i(0)&=&\xi_1^i.\\
\end{array}
\right.
\end{array}
\end{equation}
We call Equation (\ref{MEANFIELD}) the {\it mean field system}. The corresponding limiting objective functional for the $i$-th player is
\begin{equation}
\begin{array}{rcl}
	\mathcal{J}^{i}(u_1^i)=\mathbb{E}\bigg[\displaystyle\int_0^Tf_1\Big(x_1^i(t),x_0(t),m_{x_0(t)},u_1^i(t)\Big)dt+h_1\Big(x_1^i(T),x_0(T),m_{x_0(T)}\Big)\bigg].
\end{array}
\end{equation}
Using Corollary \ref{P1P2}, the necessary condition for optimality is described by the SHJB-FP coupled equation (\ref{HJBFP}). To proceed, we assume that the optimal control $\hat {\bf u}=(\hat u_1^1,\hat u_1^2,\dots,\hat u_1^N)$ exists. To avoid ambiguity, denote $\hat{x}_1^i$ and $\hat{y}_1^i$ the states dynamics of $x_1^i$ and $y_1^i$ corresponding to the optimal control $\hat{u}_1^i$. The mean field term $m_{x_0(t)}$, is the probability measure of the optimal trajectory $\hat{x}_1^i$ at time $t$, conditioning on $\mathcal{F}^0_t$. Under this construction, being conditional on $\mathcal{F}^0_t$, $\{\hat x_1^i\}_i$ are identical and independent processes; while $\{\hat{y}_1^i\}_i$ are dependent on each other through the empirical distribution. For simplicity, for two density functions $m$ and $m'$, we write $W_2(m d \lambda,m' d \lambda)=W_2(m,m')$. 
\begin{lem}
\label{1stestimate}
	Suppose the assumptions (A.1-A.3) hold. If $m_{x_0(t)}$ is chosen to be the density function of $\hat x_1^i$ conditional on $\mathcal{F}_0^t$, then
	\begin{equation*}
		\mathbb{E}\Big[\displaystyle\sup_{u \leq T}|y_0(u)-x_0(u)|^2\Big]+\mathbb{E}\Big[\displaystyle\sup_{u \leq T}|\hat{y}_1^i(u)-\hat{x}_1^i(u)|^2\Big]=\mathcal{O}(\dfrac{1}{N}).\\
	\end{equation*}
\end{lem}
\begin{proof}
	Observe that for any $t\in[0,T]$
	\begin{equation*}
	\begin{array}{rcl}
		\mathbb{E}\displaystyle\sup_{u\leq t}|y_0(u)-x_0(u)|^2&\leq& C \bigg\{t\mathbb{E}\displaystyle\int_0^t \bigg|g_0\Big(y_0(s),\dfrac{1}{N}\displaystyle\sum_{j=1}^N \delta_{\hat y_1^j(s)},u_0(s)\Big)-g_0\Big(x_0(s),m_{x_0(s)},u_0(s)\Big)\bigg|^2ds\\
		&&\qquad\qquad+\mathbb{E}\displaystyle\int_0^t\bigg|\sigma_0\Big(y_0(s)\Big)-\sigma_0\Big(x_0(s)\Big)\bigg|^2ds\bigg\},
	\end{array}
	\end{equation*}
	and 
	\begin{equation*}
	\begin{array}{rcl}
		\mathbb{E}\displaystyle\sup_{u\leq t}|\hat y_1^i(u)-\hat x_1^i(u)|^2&\leq& C \bigg\{t\mathbb{E}\displaystyle\int_0^t \bigg|g_1\Big(\hat y_1^i(s),y_0(s),\dfrac{1}{N-1}\displaystyle\sum_{j=1,j\neq i}^N \delta_{\hat y_1^j(s)},\hat u_1^i(s)\Big)-g_1\Big(\hat x_1^i(s),x_0(s),m_{x_0(s)},\hat u_1^i(s)\Big)\bigg|^2ds\\
		&&\qquad\qquad+\mathbb{E}\displaystyle\int_0^t\bigg|\sigma_1\Big(\hat y_1^i(s)\Big)-\sigma_1\Big(\hat x_1^i(s)\Big)\bigg|^2ds\bigg\},
	\end{array}
	\end{equation*}
	By the Lipschitz assumptions, we have
	\begin{equation}
	\label{CONV1}
	\begin{array}{rcl}
		&&\mathbb{E}\displaystyle\sup_{u\leq t}|y_0(u)-x_0(u)|^2+\mathbb{E}\displaystyle\sup_{u\leq t}|\hat y_1^i(u)-\hat x_1^i(u)|^2\\
		&\leq&C\bigg\{t\mathbb{E}\displaystyle\int_0^t \Big|y_0(s)-x_0(s)\Big|^2+W_2^2\Big(\dfrac{1}{N}\displaystyle\sum_{j=1}^N \delta_{\hat y_1^j(s)},m_{x_0(s)}\Big)ds+\mathbb{E}\displaystyle\int_0^t\Big|y_0(s)-x_0(s)\Big|^2 ds\bigg\}\\
		&&\qquad+C\bigg\{t\mathbb{E}\displaystyle\int_0^t \Big|\hat y_1^i(s)-\hat x_1^i(s)\Big|^2+\Big|y_0(s)-x_0(s)\Big|^2+W_2^2\Big(\dfrac{1}{N-1}\displaystyle\sum_{j=1,j\neq i}^N \delta_{\hat y_1^j(s)},m_{x_0(t)}\Big)ds+\mathbb{E}\displaystyle\int_0^t\Big|y_1^i(s)-x_1^i(s)\Big|^2 ds\bigg\}\\
		&\leq&C\bigg\{\mathbb{E}\displaystyle\int_0^t \displaystyle\sup_{u\leq s}\Big|y_0(u)-x_0(u)\Big|^2+\displaystyle\sup_{u\leq s}\Big|\hat y_1^i(u)-\hat x_1^i(u)\Big|^2\\
		&&\qquad+W_2^2\Big(\dfrac{1}{N}\displaystyle\sum_{j=1}^N \delta_{\hat y_1^j(s)},\dfrac{1}{N}\displaystyle\sum_{j=1}^N \delta_{\hat x_1^j(s)}\Big)+W_2^2\Big(\dfrac{1}{N}\displaystyle\sum_{j=1}^N \delta_{\hat x_1^j(s)},m_{x_0(s)}\Big)\\
		&&\qquad+W_2^2\Big(\dfrac{1}{N-1}\displaystyle\sum_{j=1,j\neq i}^N \delta_{\hat y_1^j(s)},\dfrac{1}{N-1}\displaystyle\sum_{j=1,j\neq i}^N \delta_{\hat x_1^j(s)}\Big)+W_2^2\Big(\dfrac{1}{N-1}\displaystyle\sum_{j=1,j\neq i}^N \delta_{\hat x_1^j(s)},m_{x_0(s)}\Big)ds\bigg\},\\
	\end{array}
	\end{equation}
	where $C>0$ is a constant, changing line by line, depends only on $T$ and $K$. By definition, for any Dirac measures $\delta_y$ on $\mathbb{R}^{n_1}$ and density function $m$, we have
		\begin{equation*}
			W_2^2(\delta_y,m)=\int_{\mathbb{R}^{n_1}}|y-x|^2 dm(x).
		\end{equation*}	
		Also observe that the joint measure $\frac{1}{N}\sum_{j=1}^N\delta_{(\hat y_1^j(s), \hat x_1^j(s))}$ on ${\mathbb{R}^{n_1}\times{\mathbb{R}^{n_1}}}$ has respective marginals $\frac{1}{N}\sum_{j=1}^N\delta_{\hat y_1^j(s)}$ and $\frac{1}{N}\sum_{j=1}^N\delta_{\hat x_1^j(s)}$ on ${\mathbb{R}^{n_1}}$. Using the definition of Wasserstein metric, we evaluate
	\begin{equation}\label{W2_GW}
	\begin{array}{rcl}
		\mathbb{E}\bigg[W_2^2\Big(\dfrac{1}{N}\displaystyle\sum_{j=1}^N \delta_{\hat y_1^j(s)},\dfrac{1}{N}\displaystyle\sum_{j=1}^N \delta_{\hat x_1^j(s)}\Big)\bigg]&\leq&\mathbb{E}\bigg[\displaystyle\int_{\mathbb{R}^{n_1}\times{\mathbb{R}^{n_1}}}|y-x|^2 d\bigg(\dfrac{1}{N}\sum_{j=1}^N\delta_{(\hat y_1^j(s), \hat x_1^j(s))}(y,x)\bigg)\bigg]\\
		&\leq&\dfrac{1}{N}\displaystyle\sum_{j=1}^N\mathbb{E}\Big|\hat y_1^j(s)-\hat x_1^j(s)\Big|^2\\
		&=&\mathbb{E}\Big|\hat y_1^i(s)-\hat x_1^i(s)\Big|^2,
	\end{array}
	\end{equation}
	where the last equality results from the fact that $\{\hat y^j - \hat x^j\}_{j=1}^N$ are symmetric. Similarly, we also have
	\begin{equation*}
	\begin{array}{rcl}
		\mathbb{E}\bigg[W_2^2\Big(\dfrac{1}{N-1}\displaystyle\sum_{j=1,j\neq i}^N \delta_{\hat y_1^j(s)},\dfrac{1}{N-1}\displaystyle\sum_{j=1,j\neq i}^N \delta_{\hat x_1^j(s)}\Big)\bigg]&\leq&\mathbb{E}\Big|\hat y_1^i(s)-\hat x_1^i(s)\Big|^2.
	\end{array}
	\end{equation*}	
	Combining with (\ref{CONV1}) and applying Gronwall's inequality, we have
	\begin{equation}
	\label{gw}
	\begin{array}{rcl}
		&&\mathbb{E}\displaystyle\sup_{u\leq t}|y_0(u)-x_0(u)|^2+\mathbb{E}\displaystyle\sup_{u\leq t}|\hat y_1^i(u)-\hat x_1^i(u)|^2\\
		&\leq&Ce^{Ct}\mathbb{E}\bigg[\displaystyle\int_0^t W_2^2\Big(\dfrac{1}{N}\displaystyle\sum_{j=1}^N \delta_{\hat x_1^j(s)},m_{x_0(s)}\Big)+W_2^2\Big(\dfrac{1}{N-1}\displaystyle\sum_{j=1,j\neq i}^N \delta_{\hat x_1^j(s)},m_{x_0(s)}\Big)ds\bigg].\\
	\end{array}
	\end{equation}	
	By definition of the Wasserstein metric, we have
	\begin{equation*}
	\begin{array}{rcl}
		&&W_2^2\Big(\dfrac{1}{N}\displaystyle\sum_{j=1}^N \delta_{\hat x_1^j(s)},m_{x_0(s)}\Big)\\
		&=&\displaystyle\inf_{\Gamma}\displaystyle\int_{\mathbb{R}^{n_1}}\displaystyle\int_{\mathbb{R}^{n_1}}|x-y|^2d\Gamma_{\Big(\frac{1}{N}\sum_{j=1}^N \delta_{\hat x_1^j(s)},m_{x_0(s)}\Big)}(x,y)\\
		&\leq&\displaystyle\int_{\mathbb{R}^{n_1}}\displaystyle\int_{\mathbb{R}^{n_1}}|x|^2-2x\cdot y +|y|^2d\Big(\frac{1}{N}\sum_{j=1}^N \delta_{\hat x_1^j(s)}\Big)(x)dm_{x_0(s)}(y)\\
		&=&\Big|\dfrac{1}{N}\displaystyle\sum_{j=1}^N \hat x_1^j(s)\Big|^2-2\Big(\dfrac{1}{N}\displaystyle\sum_{j=1}^N \hat x_1^j(s)\Big)\cdot \mathbb{E}^{\mathcal{F}^0_s}\hat x_1^j(s)+\mathbb{E}^{\mathcal{F}^0_s}|\hat x_1^j(s)|^2\\
		&=&\Big|\dfrac{1}{N}\displaystyle\sum_{j=1}^N \hat x_1^j(s)-\mathbb{E}^{\mathcal{F}^0_s}\hat x_1^j(s)\Big|^2\\
	\end{array}
	\end{equation*}
	Hence, 
	\begin{equation*}
	\begin{array}{rcl}
		&&\mathbb{E}\bigg[W_2^2\Big(\dfrac{1}{N}\displaystyle\sum_{j=1}^N \delta_{\hat x_1^j(s)},m_{x_0(s)}\Big)\bigg]\\
		&\leq&\dfrac{1}{N^2}\mathbb{E}\bigg[\displaystyle\sum_{j=1}^N \Big|\hat x_1^j(s)-\mathbb{E}^{\mathcal{F}^0_s}\hat x_1^j(s)\Big|^2+2\displaystyle\sum_{j<k}^N[\hat x_1^j(s)-\mathbb{E}^{\mathcal{F}^0_s}\hat x_1^j(s)]\cdot[\hat x_1^k(s)-\mathbb{E}^{\mathcal{F}^0_s}\hat x_1^k(s)]\bigg]\\
	\end{array}
	\end{equation*}
	Recall that given $\mathcal{F}^0_t$, $\{\hat x_1^j\}_j$ are identically and independently distributed, we thus get
	\begin{equation*}
	\begin{array}{rcl}
		&&\mathbb{E}\bigg[W_2^2\Big(\dfrac{1}{N}\displaystyle\sum_{j=1}^N \delta_{\hat x_1^j(s)},m_{x_0(s)}\Big)\bigg]\\
		&\leq&\dfrac{1}{N^2}\mathbb{E}\bigg[\displaystyle\sum_{j=1}^N \Big|\hat x_1^j(s)-\mathbb{E}^{\mathcal{F}^0_s}\hat x_1^j(s)\Big|^2+2\displaystyle\sum_{j<k}^N[\mathbb{E}^{\mathcal{F}^0_s}\hat x_1^j(s)-\mathbb{E}^{\mathcal{F}^0_s}\hat x_1^j(s)]\cdot[\mathbb{E}^{\mathcal{F}^0_s}\hat x_1^k(s)-\mathbb{E}^{\mathcal{F}^0_s}\hat x_1^k(s)]\bigg]\\
		&=&\dfrac{1}{N^2}\mathbb{E}\bigg[\displaystyle\sum_{j=1}^N \Big|\hat x_1^j(s)-\mathbb{E}^{\mathcal{F}^0_s}\hat x_1^j(s)\Big|^2+0\bigg]\\
		&=&\dfrac{1}{N}\mathbb{E}\Big|\hat x_1^j(s)-\mathbb{E}^{\mathcal{F}^0_s}\hat x_1^j(s)\Big|^2.\\
	\end{array}
	\end{equation*}	
	Similar estimate applies on the second term in Equation (\ref{gw}). Put $t=T$, we finally have
	\begin{equation}
	\label{gw2}
	\begin{array}{rcl}
		&&\mathbb{E}\displaystyle\sup_{u\leq T}|y_0(u)-x_0(u)|^2+\mathbb{E}\displaystyle\sup_{u\leq T}|\hat y_1^i(u)-\hat x_1^i(u)|^2\\
		&\leq&\dfrac{Ce^{CT}}{N}\mathbb{E}\bigg[\displaystyle\int_0^T \Big|\hat x_1^j(s)-\mathbb{E}^{\mathcal{F}^0_s}\hat x_1^j(s)\Big|^2ds\bigg]\\
		&\leq&\dfrac{4Ce^{CT}}{N}\mathbb{E}\bigg[\displaystyle\int_0^T \Big|\hat x_1^j(s)\Big|^2ds\bigg]
	\end{array}
	\end{equation}
	With the linear growth assumptions, for any $t\in[0,T]$, we easily get the estimates
	\begin{equation*}
	\begin{array}{rcl}
		\mathbb{E}\displaystyle\sup_{u\leq t}|x_0(t)|^2\leq C\mathbb{E}\bigg\{|\xi_0|^2+\displaystyle\int_0^t \displaystyle\sup_{u\leq s}|x_0(u)|^2+\mathbb{E}^{\mathcal{F}_s^0}|\hat{x}_1^j(s)|^2+|u_0(s)|^2ds\bigg\}
	\end{array}
	\end{equation*}	
	and
	\begin{equation*}
	\begin{array}{rcl}
		\mathbb{E}\displaystyle\sup_{u\leq t}|\hat{x}_1^j(t)|^2\leq C\mathbb{E}\bigg\{|\xi_1^j|^2+\displaystyle\int_0^t \displaystyle\sup_{u\leq s}|\hat{x}_1^j(u)|^2+\displaystyle\sup_{u\leq s}|x_0(u)|^2+\mathbb{E}^{\mathcal{F}_s^0}|\hat{x}_1^j(s)|^2+|\hat{u}_1^j(s)|^2ds\bigg\}
	\end{array}
	\end{equation*}
	Applying the Tower property of expectation and the Gronwall's inequality on the sum of the two inequalities above yields
	\begin{equation*}
	\begin{array}{rcl}
		\mathbb{E}\displaystyle\sup_{u\leq t}|x_0(t)|^2+\mathbb{E}\displaystyle\sup_{u\leq t}|\hat{x}_1^j(t)|^2\leq Ce^{Ct}\mathbb{E}\bigg\{|\xi_0|^2+|\xi_1^j|^2+\displaystyle\int_0^t |\hat{u}_1^j(s)|^2+|u_0(s)|^2ds\bigg\}<\infty
	\end{array}
	\end{equation*}
	We have the order for the estimate (\ref{gw2})
	\begin{equation}
		\mathbb{E}\displaystyle\sup_{u\leq T}|y_0(u)-x_0(u)|^2+\mathbb{E}\displaystyle\sup_{u\leq T}|\hat y_1^i(u)-\hat x_1^i(u)|^2=\mathcal{O}(\dfrac{1}{N}).
	\end{equation}	
 \qed \end{proof}
We also have approximation for the cost functionals.
\begin{lem}
	\begin{equation*}
	\begin{array}{rcl}
		&&\mathcal{J}^{N,i}(\hat {\bf u})-\mathcal{J}^{i}(\hat u_1^i)=O(\dfrac{1}{\sqrt N}).\\
	\end{array}
	\end{equation*}
\end{lem}
\begin{proof}
	With the quadratic assumptions (\ref{A4_0}), we have
	\begin{equation*}
	\begin{array}{rcl}
		|\mathcal{J}^{N,i}(\hat {\bf u})-\mathcal{J}^{i}(\hat u^i)|&\leq&\mathbb{E}\bigg[\displaystyle\int_0^Tf_1\Big(\hat y_1^i(t),y_0(t),\dfrac{1}{N-1}\displaystyle\sum_{j=1, j \neq i}^N\delta_{\hat y_1^j(t)},\hat u^i(t)\Big)-f_1\Big(\hat x_1^i(t),x_0(t),m_{x_0(t)},\hat u_1^i(t)\Big)dt\\
		&&\qquad\qquad+h_1\Big(\hat y_1^i(T),y_0(T),\dfrac{1}{N-1}\displaystyle\sum_{j=1, j \neq i}^N\delta_{\hat y_1^j(T)}\Big)-h_1\Big(\hat x_1^i(T),\hat x_0(T),m_{x_0(T)}\Big)\bigg]\\
		&\leq&C\mathbb{E}\bigg[\displaystyle\int_0^T \Big[1+|\hat y_1^i(t)|+|\hat x_1^i(t)|+|y_0(t)|+|x_0(t)|+\Big(\dfrac{\sum_{j=1, j \neq i}^N |\hat y_1^j(t)|^2}{N-1}\Big)^{\frac{1}{2}}+\Big(\mathbb{E}^{\mathcal{F}_t^0}|\hat x_1^i(t)|^2\Big)^{\frac{1}{2}}+2|\hat u^i(t)|\Big]\\
		&&\qquad\qquad\qquad\cdot\Big[|\hat y_1^i(t)-\hat x_1^i(t)|+|y_0(t)-x_0(t)|+W_2\Big(\dfrac{1}{N-1}\displaystyle\sum_{j=1, j \neq i}^N\delta_{\hat y_1^j(t)},m_{x_0(t)}\Big)\Big]dt\\
		&&\qquad+\Big[1+|\hat y_1^i(T)|+|\hat x_1^i(T)|+|y_0(T)|+|x_0(T)|+\Big(\dfrac{\sum_{j=1, j \neq i}^N |\hat y_1^j(T)|^2}{N-1}\Big)^{\frac{1}{2}}+\Big(\mathbb{E}^{\mathcal{F}_t^0}|\hat x_1^i(T)|^2\Big)^{\frac{1}{2}}\Big]\\
		&&\qquad\qquad\qquad\cdot\Big[|\hat y_1^i(T)-\hat x_1^i(T)|+|y_0(T)-x_0(T)|+W_2\Big(\dfrac{1}{N-1}\displaystyle\sum_{j=1, j \neq i}^N\delta_{\hat y_1^j(T)},m_{x_0(T)}\Big)\Big]\bigg].\\
	\end{array}
	\end{equation*}
	An application of H\"older's inequality, and the symmetry on $\{\hat{x}_1^j\}_j$ gives 	
	\begin{equation*}
	\begin{array}{rcl}
		|\mathcal{J}^{N,i}(\hat {\bf u})-\mathcal{J}^{i}(\hat u^i)|&\leq&C\bigg\{\bigg[\mathbb{E}\displaystyle\int_0^T \Big[1+|\hat y_1^i(t)|^2+|\hat x_1^i(t)|^2+|y_0(t)|^2+|x_0(t)|^2+\dfrac{\sum_{j=1, j \neq i}^N |\hat y_1^j(t)|^2}{N-1}+\mathbb{E}^{\mathcal{F}_t^0}|\hat x_1^i(t)|^2+|\hat u^i(t)|^2\Big]dt\bigg]^{\frac{1}{2}}\\
		&&\qquad\qquad\qquad\cdot\bigg[\mathbb{E}\displaystyle\int_0^T\Big[|\hat y_1^i(t)-\hat x_1^i(t)|^2+|y_0(t)-x_0(t)|^2+W_2^2\Big(\dfrac{1}{N-1}\displaystyle\sum_{j=1, j \neq i}^N\delta_{\hat y_1^j(t)},m_{x_0(t)}\Big)\Big]dt\bigg]^{\frac{1}{2}}\\
		&&\qquad+\bigg[\mathbb{E}\Big[1+|\hat y_1^i(T)|^2+|\hat x_1^i(T)|^2+|y_0(T)|^2+|x_0(T)|^2+\dfrac{\sum_{j=1, j \neq i}^N |\hat y_1^j(T)|^2}{N-1}+\mathbb{E}^{\mathcal{F}_t^0}|\hat x_1^i(T)|^2\Big]\bigg]^{\frac{1}{2}}\\
		&&\qquad\qquad\qquad\cdot\bigg[\mathbb{E}\Big[|\hat y_1^i(T)-\hat x_1^i(T)|^2+|y_0(T)-x_0(T)|^2+W^2_2\Big(\dfrac{1}{N-1}\displaystyle\sum_{j=1, j \neq i}^N\delta_{\hat y_1^j(T)},m_{x_0(T)}\Big)\Big]\bigg]^{\frac{1}{2}}\bigg\}\\
		&=&C\bigg\{\bigg[\mathbb{E}\displaystyle\int_0^T \Big[1+|\hat y_1^i(t)|^2+|\hat x_1^i(t)|^2+|y_0(t)|^2+|x_0(t)|^2+|\hat u^i(t)|^2\Big]dt\bigg]^{\frac{1}{2}}\\
		&&\qquad\qquad\qquad\cdot\bigg[\mathbb{E}\displaystyle\int_0^T\Big[|\hat y_1^i(t)-\hat x_1^i(t)|^2+|y_0(t)-x_0(t)|^2+W_2^2\Big(\dfrac{1}{N-1}\displaystyle\sum_{j=1, j \neq i}^N\delta_{\hat x_1^j(t)},m_{x_0(t)}\Big)\Big]dt\bigg]^{\frac{1}{2}}\\
		&&\qquad+\bigg[\mathbb{E}\Big[1+|\hat y_1^i(T)|^2+|\hat x_1^i(T)|^2+|y_0(T)|^2+|x_0(T)|^2\Big]\bigg]^{\frac{1}{2}}\\
		&&\qquad\qquad\qquad\cdot\bigg[\mathbb{E}\Big[|\hat y_1^i(T)-\hat x_1^i(T)|^2+|y_0(T)-x_0(T)|^2+W^2_2\Big(\dfrac{1}{N-1}\displaystyle\sum_{j=1, j \neq i}^N\delta_{\hat x_1^j(T)},m_{x_0(T)}\Big)\Big]\bigg]^{\frac{1}{2}}\bigg\}\\
	\end{array}
	\end{equation*}
	By the linear growth assumptions on $g_0$, $\sigma_0$, $g_1$ and $\sigma_1$, it is easy to show that 
	\begin{equation*}
		\mathbb{E}\displaystyle\int_0^T \Big[1+|\hat y_1^i(t)|^2+|\hat x_1^i(t)|^2+|y_0(t)|^2+|x_0(t)|^2+|\hat u^i(t)|^2\Big]dt
	\end{equation*}
	and
	\begin{equation*}
		\mathbb{E}\Big[1+|\hat y_1^i(T)|^2+|\hat x_1^i(T)|^2+|y_0(T)|^2+|x_0(T)|^2\Big]
	\end{equation*}
	are bounded (independent of $N$). We finally arrive at the estimates
	\begin{equation*}
	\begin{array}{rcl}
		|\mathcal{J}^{N,i}(\hat {\bf u})-\mathcal{J}^{i}(\hat u^i)|&\leq&C\bigg\{\bigg[\mathbb{E}\displaystyle\int_0^T\Big[|\hat y_1^i(t)-\hat x_1^i(t)|^2+|y_0(t)-x_0(t)|^2+W_2^2\Big(\dfrac{1}{N-1}\displaystyle\sum_{j=1, j \neq i}^N\delta_{\hat x_1^j(t)},m_{x_0(t)}\Big)\Big]dt\bigg]^{\frac{1}{2}}\\
		&&\qquad+\bigg[\mathbb{E}\Big[|\hat y_1^i(T)-\hat x_1^i(T)|^2+|y_0(T)-x_0(T)|^2+W^2_2\Big(\dfrac{1}{N-1}\displaystyle\sum_{j=1, j \neq i}^N\delta_{\hat x_1^j(T)},m_{x_0(T)}\Big)\Big]\bigg]^{\frac{1}{2}}\bigg\},\\
	\end{array}
	\end{equation*}
	which goes to $0$ as $N\rightarrow \infty$, as shown in Lemma \ref{1stestimate}. Hence
	\begin{equation*}
		|\mathcal{J}^{N,i}(\hat {\bf u})-\mathcal{J}^{i}(\hat u_1^i)|=O(\dfrac{1}{\sqrt{N}}).
	\end{equation*}
 \qed \end{proof}
In the previous lemmas, we assumed that all players adopt their corresponding mean field optimal controls. By symmetry, the convergences of state dynamics and the cost functionals are then established. To show that the mean field optimal controls ${\bf u}$ indeed constitute a $\epsilon$-Nash equilibrium on the empirical system, without loss of generality, we assume that the first player did not obey the mean field optimal control. In particular, let $u_1^1$ be an arbitrary control in $\mathcal{A}_1$, define ${\bf u}:=(u_1^1,\hat u_1^2,\dots,\hat u_1^N)$. We then have the following empirical and mean field SDEs for the dominating player, the $1$-st player and the $i$-th player ($i>1$) respectively:
\begin{equation}
\left\{
\begin{array}{rcl}
	dy_0&=&g_0\Big(y_0(t),\dfrac{1}{N}\Big(\delta_{ y_1^1(t)}+\displaystyle\sum_{j=2}^N \delta_{\hat y_1^j(t)}\Big),u_0(t)\Big)dt+\sigma_0\Big(y_0(t)\Big)dW_0(t),\\
	y_0(0)&=&\xi_0.\\
	dx_0&=&g_0\Big(x_0(t),m_{x_0(t)},u_0(t)\Big)dt+\sigma_0\Big(x_0(t)\Big)dW_0(t),\\
	x_0(0)&=&\xi_0.\\	
\end{array}
\right.
\end{equation}
\begin{equation}
\left\{
\begin{array}{rcl}
	dy_1^1&=&g_1\Big(y_1^1(t),y_0(t),\dfrac{1}{N-1}\displaystyle\sum_{j=2}^N \delta_{\hat y_1^j(t)},u_1^1(t)\Big)dt+\sigma_1\Big(y_1^1(t)\Big)dW_1^1(t),\\
	y_1^1(0)&=&\xi_1^1.\\
	dx_1^1&=&g_1\Big(x_1^1(t),x_0(t),m_{x_0(t)},u_1^1(t)\Big)dt+\sigma_1\Big(x_1^1(t)\Big)dW_1^1(t),\\
	x_1^1(0)&=&\xi_1^1.\\
\end{array}
\right.
\end{equation}
\begin{equation}
\left\{
\begin{array}{rcl}
	d\hat y_1^i&=&g_1\Big(\hat y_1^i(t),y_0(t),\dfrac{1}{N-1}\Big(\delta_{ y_1^1(t)}+\displaystyle\sum_{j=2,j\neq i}^N \delta_{\hat y_1^j(t)}\Big),\hat u_1^i(t)\Big)dt+\sigma_1\Big(\hat y_1^i(t)\Big)dW_1^i(t),\\
	\hat y_1^i(0)&=&\xi_1^i.\\
	d\hat x_1^i&=&g_1\Big(\hat x_1^i(t),x_0(t),m_{x_0(t)},\hat u_1^i(t)\Big)dt+\sigma_1\Big(\hat x_1^i(t)\Big)dW_1^i(t),\\
	\hat x_1^i(0)&=&\xi_1^i.\\
\end{array}
\right.
\end{equation}
We claim that if $m_{x_0}$ is the density function of $\hat{x}_1^i$ conditioning on $\mathcal{F}^0$, then we have the convergence $y_0\rightarrow x_0$, $y_1^1\rightarrow x_1^1$ and $\hat y_1^i \rightarrow \hat x_1^i$ in the sense of the following lemma
\begin{lem}
	\begin{equation*}
		\mathbb{E}\Big[\displaystyle\sup_{u \leq T}|y_0(u)-x_0(u)|^2\Big]+\mathbb{E}\Big[\displaystyle\sup_{u \leq T}|y_1^1(u)-x_1^1(u)|^2\Big]+\mathbb{E}\Big[\displaystyle\sup_{u \leq T}|\hat{y}_1^i(u)-\hat{x}_1^i(u)|^2\Big]\rightarrow 0, \quad \text{ as } N \rightarrow \infty.\\
	\end{equation*}
\end{lem}
\begin{proof}
We first show the convergence of the dominating player and the $i$-th player. Similar to the proof of Lemma \ref{1stestimate}, we first have
	\begin{equation}
	\label{2ndestimate}
	\begin{array}{rcl}
		&&\mathbb{E}\displaystyle\sup_{u\leq t}|y_0(u)-x_0(u)|^2+\mathbb{E}\displaystyle\sup_{u\leq t}|\hat y_1^i(u)-\hat x_1^i(u)|^2\\
		&\leq&C\bigg\{t\mathbb{E}\displaystyle\int_0^t \Big|y_0(s)-x_0(s)\Big|^2+W_2^2\Big(\dfrac{1}{N}\Big(\delta_{ y_1^1(s)}+\displaystyle\sum_{j=2}^N \delta_{\hat y_1^j(s)}\Big),m_{x_0(s)}\Big)ds+\mathbb{E}\displaystyle\int_0^t\Big|y_0(s)-x_0(s)\Big|^2\bigg\}ds\\
		&&\qquad+C\bigg\{t\mathbb{E}\displaystyle\int_0^t \Big|\hat y_1^i(s)-\hat x_1^i(s)\Big|^2+\Big|y_0(s)-x_0(s)\Big|^2+W_2^2\Big(\dfrac{1}{N-1}\Big(\delta_{ y_1^1(s)}+\displaystyle\sum_{j=2,j\neq i}^N \delta_{\hat y_1^j(s)}\Big),m_{x_0(s)}\Big)ds+\mathbb{E}\displaystyle\int_0^t\Big|\hat y_1^i(s)-\hat x_1^i(s)\Big|^2\bigg\}ds\\
		&\leq&C\mathbb{E}\displaystyle\int_0^t \bigg[\displaystyle\sup_{u\leq s}\Big|y_0(u)-x_0(u)\Big|^2+\displaystyle\sup_{u\leq s}\Big|\hat y_1^i(u)-\hat x_1^i(u)\Big|^2\\
		&&\qquad+W_2^2\Big(\dfrac{1}{N}\Big(\delta_{ y_1^1(s)}+\displaystyle\sum_{j=2}^N \delta_{\hat y_1^j(s)}\Big),\dfrac{1}{N-1}\displaystyle\sum_{j=2}^N \delta_{\hat y_1^j(s)}\Big)+W_2^2\Big(\dfrac{1}{N-1}\displaystyle\sum_{j=2}^N \delta_{\hat y_1^j(s)},\dfrac{1}{N-1}\displaystyle\sum_{j=2}^N \delta_{\hat x_1^j(s)}\Big)+W_2^2\Big(\dfrac{1}{N-1}\displaystyle\sum_{j=2}^N \delta_{\hat x_1^j(s)},m_{x_0(s)}\Big)\\
		&&\qquad+W_2^2\Big(\dfrac{1}{N-1}\Big(\delta_{ y_1^1(s)}+\displaystyle\sum_{j=2,j\neq i}^N \delta_{\hat y_1^j(s)}\Big),\dfrac{1}{N-2}\displaystyle\sum_{j=2,j\neq i}^N \delta_{\hat y_1^j(s)}\Big)+W_2^2(\dfrac{1}{N-2}\displaystyle\sum_{j=2,j\neq i}^N \delta_{\hat y_1^j(s)},\dfrac{1}{N-2}\displaystyle\sum_{j=2,j\neq i}^N \delta_{\hat x_1^j(s)})\\
		&&\qquad+W_2^2\Big(\dfrac{1}{N-2}\displaystyle\sum_{j=2,j\neq i}^N \delta_{\hat x_1^j(s)},m_{x_0(s)}\Big)\bigg]ds\\
	\end{array}
	\end{equation}
	By the same argument used in Equation (\ref{W2_GW}) in Lemma \ref{1stestimate}, we have
	\begin{equation*}
		\mathbb{E}\Big[W_2^2\Big(\dfrac{1}{N-1}\displaystyle\sum_{j=2}^N \delta_{\hat y_1^j(s)},\dfrac{1}{N-1}\displaystyle\sum_{j=2}^N \delta_{\hat x_1^j(s)}\Big)\Big]+\mathbb{E}\Big[W_2^2(\dfrac{1}{N-2}\displaystyle\sum_{j=2,j\neq i}^N \delta_{\hat y_1^j(s)},\dfrac{1}{N-2}\displaystyle\sum_{j=2,j\neq i}^N \delta_{\hat x_1^j(s)})\Big]\leq 2\mathbb{E}|\hat y_1^i(s)-\hat x_1^i|^2.
	\end{equation*}
	Hence, by applying Gronwall's inequality on Equation (\ref{2ndestimate}), we have 
	\begin{equation}
	\label{3rdestimate}
	\begin{array}{rcl}
		&&\mathbb{E}\displaystyle\sup_{u\leq t}|y_0(u)-x_0(u)|^2+\mathbb{E}\displaystyle\sup_{u\leq t}|\hat y_1^i(u)-\hat x_1^i(u)|^2\\
		&\leq&Ce^{Ct}\mathbb{E}\displaystyle\int_0^t \bigg[W_2^2\Big(\dfrac{1}{N}\Big(\delta_{ y_1^1(s)}+\displaystyle\sum_{j=2}^N \delta_{\hat y_1^j(s)}\Big),\dfrac{1}{N-1}\displaystyle\sum_{j=2}^N \delta_{\hat y_1^j(s)}\Big)+W_2^2\Big(\dfrac{1}{N-1}\displaystyle\sum_{j=2}^N \delta_{\hat x_1^j(s)},m_{x_0(s)}\Big)\\
		&&\qquad+W_2^2\Big(\dfrac{1}{N-1}\Big(\delta_{ y_1^1(s)}+\displaystyle\sum_{j=2,j\neq i}^N \delta_{\hat y_1^j(s)}\Big),\dfrac{1}{N-2}\displaystyle\sum_{j=2,j\neq i}^N \delta_{\hat y_1^j(s)}\Big)+W_2^2\Big(\dfrac{1}{N-2}\displaystyle\sum_{j=2,j\neq i}^N \delta_{\hat x_1^j(s)},m_{x_0(s)}\Big)\bigg]ds\\
	\end{array}
	\end{equation} 
	For the first term in (\ref{3rdestimate}), consider the following joint measure on $\mathbb{R}^{n_1}\times \mathbb{R}^{n_1}$
	\begin{equation*}
	\begin{split}
	\mu(x,y)&=\dfrac{1}{N}\displaystyle\sum_{j=2}^N\delta_{(\hat y_1^j(s),\hat y_1^j(s))}(x,y)+\dfrac{1}{N(N-1)}\displaystyle\sum_{j=2}^N\delta_{(y_1^1(s),\hat y_1^j(s))}(x,y),\\
	\end{split}
	\end{equation*}		
which has respective marginals
	\begin{equation*}
		\dfrac{1}{N}\Big(\delta_{ y_1^1(s)}+\displaystyle\sum_{j=2}^N \delta_{\hat y_1^j(s)}\Big)\quad\text{and}\quad \dfrac{1}{N-1}\displaystyle\sum_{j=2}^N \delta_{\hat y_1^j(s)}.
	\end{equation*}
By the definition of Wasserstein metric,
	\begin{equation*}
	\begin{array}{rcl}
		&&\mathbb{E}\Big[W_2^2\Big(\dfrac{1}{N}\Big(\delta_{ y_1^1(s)}+\displaystyle\sum_{j=2}^N \delta_{\hat y_1^j(s)}\Big),\dfrac{1}{N-1}\displaystyle\sum_{j=2}^N \delta_{\hat y_1^j(s)}\Big)\Big]\\
		&\leq&\mathbb{E}\Big[\displaystyle\int_{\mathbb{R}^{n_1}\times\mathbb{R}^{n_1}}|x-y|^2d\mu(x,y)\Big]\\
		&=&\mathbb{E}\Big[\dfrac{1}{N(N-1)}\displaystyle\sum_{j=2}^N|y_1^1(s)-\hat y_1^j(s)|^2\Big]\\
		&=&\dfrac{1}{N}\mathbb{E}|y_1^1(s)-\hat y_1^2(s)|^2,
	\end{array}
	\end{equation*}
	where the last equality results from symmetry on $\{y_1^j\}_j$, clearly goes to $0$ as $N\rightarrow \infty$. Similar argument applies for the third term in (\ref{3rdestimate}). For the convergence of the second and the forth term, we refer to the argument in the last part of Lemma \ref{1stestimate} and the results follow.
	
	For the convergence of the $1$-st player, the procedure are similar and we do not provide here.
\qed \end{proof}

We conclude from the similar procedures to show the convergence of the cost functional. In particular, we have
	\begin{equation*}
		|\mathcal{J}^{N,1}({\bf u})-\mathcal{J}^{1}(u_1^1)|=O(\dfrac{1}{\sqrt{N}}).
	\end{equation*}
\begin{thm}
	$\hat {\bf u}$ is an $\epsilon$-Nash equilibrium.
\end{thm}
\begin{proof}
	Summarizing all the obtained results in this section, we can conclude 
	\begin{equation*}
	\begin{array}{rcl}
		&&|\mathcal{J}^{N,1}(\hat {\bf u})-\mathcal{J}^{1}(\hat u_1^1)|=O(\dfrac{1}{\sqrt N});\\
		&&|\mathcal{J}^{N,1}({\bf u})-\mathcal{J}^{1}( u_1^1)|=O(\dfrac{1}{\sqrt N}).\\
	\end{array}
	\end{equation*}
	Since $\hat u_1^1$ is optimal control, we have $\mathcal{J}^{1}(\hat u_1^1) \leq \mathcal{J}^{1}( u_1^1)$. We deduce
	\begin{equation*}
		\mathcal{J}^{N,i}(\hat {\bf u}) \leq \mathcal{J}^{N,1}({\bf u})+O(\dfrac{1}{\sqrt N}).
	\end{equation*}
	Hence, $\hat {\bf u}$ is an $\epsilon$-Nash equilibrium.
 \qed \end{proof}

\end{document}